\documentclass[12pt]{amsart}
\usepackage{amssymb}
\usepackage{amscd}
\usepackage{amsmath,amssymb}

\usepackage{graphicx}

\theoremstyle{plain}
\newtheorem{thm}{Theorem}
\newtheorem{cor}{Corollary}
\newtheorem{lem}{Lemma}[subsection]
\newtheorem{prop}{Proposition}[subsection]
\theoremstyle{definition}
\newtheorem{defn}{Definition}[subsection]
\newtheorem{rem}{Remark}

\renewcommand{\int}{\operatorname{int}}

\newcommand{\arf}{\mathrm{Arf}}

\newcommand{\Z}{\mathbb{Z}}
\newcommand{\N}{\mathbb{N}}

\newcommand{\cT}{\mathcal{T}}
\newcommand{\cW}{\mathcal{W}}

\begin{document}

\title{Stable concordance of knots in 3--manifolds}

\author{Rob Schneiderman}
\address{Department of Mathematics and Computer Science,
Lehman College, \\City University of New York, New York NY}

\begin{abstract}
Knots and links in 3-manifolds are studied by applying
intersection invariants to singular concordances. The resulting
link invariants generalize the Arf invariant, the mod 2
Sato-Levine invariants, and Milnor's triple linking numbers.
Besides fitting into a general theory of Whitney towers, these
invariants provide obstructions to the existence of a singular
concordance which can be homotoped to an embedding after
stabilization by connected sums with $S^2\times S^2$. Results
include classifications of stably slice links in orientable
3-manifolds, stable knot concordance in products of an orientable
surface with the circle, and stable link concordance for many
links of null-homotopic knots in orientable 3-manifolds.

\end{abstract}


\maketitle

\section{Introduction}
A generic homotopy of a knotted circle in a 3-manifold $M$ traces
out a properly immersed annulus in the product $M\times I$ of $M$
with an interval. Knot invariants can be extracted from homotopy
invariants of such a singular concordance by factoring out the
contributions from self-homotopies of knots. For instance, the
(relative) self-linking invariants of \cite{S} are determined by
the self-intersection numbers of immersed annuli in the group ring
$\Z[\pi_1M]$, and are characterized as the complete obstruction to
finding an ``improved'' singular concordance in the sense that all
of its singularities are paired by Whitney disks. In this setting,
we will describe how a generalization of the intersection
invariant of \cite{ST1} can be used to define ``higher order''
linking invariants which give obstructions to building a ``higher
order Whitney tower'', meaning that intersections between annuli
and the interiors of previously existing Whitney disks can be
paired by a second layer of Whitney disks. As a corollary, we can
in many cases determine which links can cobound a collection $A$
of properly immersed annuli in $M\times I$ such that $A$ is
homotopic (rel boundary) to an embedding after stabilization by
taking connected sums with copies of $S^2\times S^2$. If one end
of such a \emph{stable concordance} is the unlink then it can be
capped off to form \emph{stable slice} disks.

Besides detecting many infinite families of non-concordant knots
and links, the invariants defined here will give:
\begin{itemize}
  \item A complete characterization of links in orientable
  3-manifolds which are stably slice.
  \item A classification of stable concordance for many
  links of null-homotopic knots in orientable 3-manifolds with torsion-free fundamental group.
  \item A classification of stable concordance for knots
  in products $F\times S^1$ of an orientable surface with the
  circle.
\end{itemize}
The approach taken here fits into the developing theory of Whitney
towers in 4-manifolds (\cite{CST,CST2,S1,S2,ST1,ST2}), but of
equal interest in the current setting is how the topology of a
3-manifold interacts with its knot theory. This interaction is
reflected in the indeterminacies in the invariants coming from
self-homotopies of knots, which can be studied using classical
3-dimensional structure theorems, following the approach of
\cite{K,KL1,KL2}. It turns out that these indeterminacies depend
only on lower order invariants which can often be easily computed
in the 3-manifold.

The algebraic linking invariants of \cite{S} are defined for all
knots in a large family of orientable 3-manifolds. For a link of
null-homotopic knots the invariants reduce to the usual
equivariant linking form, but for knots in non-trivial (free)
homotopy classes these invariants are \emph{relative} invariants,
which compare homotopic knots to a chosen ``base knot''. One of
the main observations of \cite{S} is that although such a choice
of base knot is not canonical, there are certain preferred knots
which ``maximize'' the geometric characterization of the
invariants by ``minimizing'' the indeterminacies. These knots are
called \emph{spherical knots} in \cite{S} because all their
corresponding indeterminacies come from intersections between the
knots and 2-spheres in the 3-manifold. (In general, knot
self-homotopies project to tori in $M$ which can contribute
indeterminacies even when $M$ is irreducible.) It is only with
respect to spherical knots that the invariants of \cite{S}
classify \emph{order 1 Whitney concordance}, i.e. give the
complete obstruction to the existence of a singular concordance
having all double points paired by Whitney disks.

In the present paper we take a more streamlined approach by
directly prescribing the analogous ``indeterminacy-minimizing''
base knots (and links) which give the desired classification of
stable concordance and equivalent characterizations in terms of
Whitney towers. Note that we now have many more equivalence
classes to examine; for instance in a non-simply-connected
3-manifold there are infinitely many order 1 Whitney concordance
classes of null-homotopic knots. We will describe how to choose
appropriate knots so that the resulting indeterminacies only
depend on the given equivalence class. These indeterminacies will
be computable in terms of intersections between knots and
2-spheres, disks, and annuli in $M$.

The rest of this introduction will introduce enough terminology to
state the main results, leaving detailed definitions and proofs to
the body of the paper. The reader should note that there are some
essentially cosmetic differences in notation between this paper
and the closely related papers \cite{S} and \cite{ST1}. The
notation used here is chosen to fit with the more recent
literature on Whitney towers, and clarifications will be made as
terminology is introduced. While a familiarity with \cite{S} and
\cite{ST1} is certainly helpful, the present paper is to a large
extent self-contained, including sketches of much relevant
background material. In particular, properties of Whitney disks
are described in detail, and also well-illustrated in the
examples.

We work in the smooth oriented category, with orientations usually
suppressed from notation. The unit interval $I=[0,1]$ will be
reparametrized without mention. When possible we follow the
convention of knot theory to confuse an immersion with its image.


\subsection{Intersection trees and stable embeddings}
\label{subsec:intro-low-order-ints-stable-embeddings} The homotopy
invariant denoted by $\tau$ in \cite{ST1} was defined for immersed
2-spheres in 4-manifolds with vanishing Wall intersection
invariant by counting ``higher order'' intersections between
Whitney disks and the 2-spheres. This invariant has since been
shown to fit into an obstruction theory for immersed surfaces in
4-manifolds in terms of Whitney towers, where an order $n$ Whitney
tower $\cW$ has an order $n$ obstruction $\tau_n(\cW)$ which lives
in an abelian group generated by trivalent trees decorated by
fundamental group elements (\cite{CST,CST2,S1,S2,ST2}). The
vanishing of $\tau_n(\cW)$ implies the existence of an order $n+1$
Whitney tower on the underlying immersed surfaces. As discussed
next and described in detail in Section~\ref{sec:int-invariants},
in this language $\tau_0(\cW)$ corresponds to Wall's intersection
form, and $\tau_1(\cW)$ corresponds to the $\tau$ in \cite{ST1}.

Let $A=A_1,A_2,\cdots,A_m$ be a collection of properly immersed
simply connected compact oriented surfaces in an oriented
4-manifold $X$. Such an $A$ is called a \emph{Whitney tower of
order zero}. Wall's intersection $\lambda(A_i,A_j)$ and
self-intersection $\mu(A_i)$ invariants count signed elements of
$\pi_1X$ associated to sheet-changing paths through the
intersection points of $A$. The equivalent \emph{order zero
intersection tree} $\tau_0(A)$ takes values in a group
$\cT_0(\pi_1X)$ of decorated order zero trees (oriented edges
labelled by elements of $\pi_1X$), and gives the complete
obstruction to the existence of an \emph{order 1 Whitney tower}
$\cW$ on $A$, where $\cW$ consists of the $A_i$ together with a
collection of Whitney disks pairing all the intersections and
self-intersections among the $A_i$.

For $A$ admitting such an order one Whitney tower $\cW$, \emph{the
order one intersection tree} $\tau_1(A):=\tau_1(\cW)$ takes values
in an abelian group $\cT_1(\pi_1X)/\mbox{INT}(A)$ generated by
decorated order one trees (unitrivalent Y-trees having a single
trivalent vertex). Here each Y-tree $t_p\in\tau_1(\cW)$ corresponds
to a transverse intersection point $p$ between an $A_k$ and the
interior of a Whitney disk $W_{(i,j)}$ pairing intersections between
$A_i$ and $A_j$. The univalent vertices of $t_p$ inherit the labels
$i$, $j$, and $k$; and the oriented edges of $t_p$ are decorated by
elements of $\pi_1X$ determined by sheet-changing paths through
$W_{(i,j)}$ (see Figure~\ref{Y-wdisk-labelled-with-tree} and
Section~\ref{subsec:tau1-def} below). The target group includes
\emph{intersection relations} INT($A$) which are determined by the
order zero invariants $\lambda(A_i,S)$ where $S$ ranges over
immersed 2-spheres representing generators of $\pi_2X$.
\begin{figure}
\centerline{\includegraphics[width=100mm]{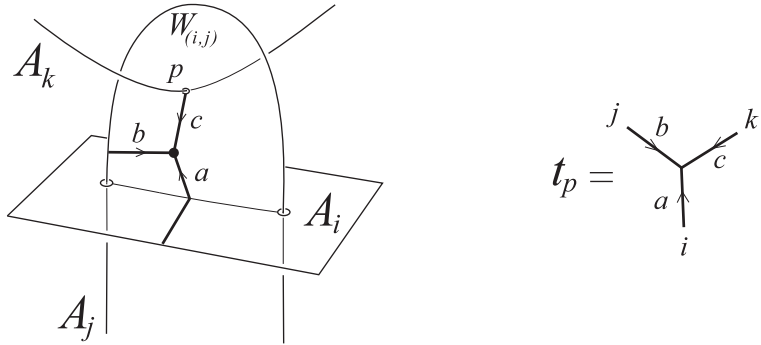}}
        \caption{A Whitney disk $W_{(i,j)}$ pairing intersections between surface sheets
         $A_i$ and $A_j$, and the Y-tree $t_p$ associated to an intersection point $p$
         between the interior of $W_{(i,j)}$
         and $A_k$.}\label{Y-wdisk-labelled-with-tree}
\end{figure}

As shown in \cite{ST1}, $\tau_1(A)$ does not depend on the choice
of $\cW$, and has the following geometric characterizations:
\begin{thm}[\cite{ST1}]\label{thm:tau}
The following are equivalent:
\begin{enumerate}
  \item $\tau_1(A)$ vanishes.
  \item $A$ is homotopic (rel boundary) to $A'$ admitting an order 2 Whitney
  tower.
  \item $A$ is homotopic (rel boundary) to $A'$ admitting a height 1 Whitney tower.
\end{enumerate}
\end{thm}
Here an \emph{order 2 Whitney tower} includes a second layer of
Whitney disks pairing all intersections between the $A_i$ and the
first layer of Whitney disks; and a \emph{height 1 Whitney tower} is
an order 1 Whitney tower such that all the interiors of the Whitney
disks are disjoint from the $A_i$. This geometric characterization
of $\tau_1$ corresponds to the fact that all the relations in the
target are realizable by controlled manipulations of the Whitney
tower.

\begin{defn}\label{defn:stable-embedding}
A collection $A$ of properly immersed surfaces $A_i$ in a
4-manifold $X$ can be \emph{stably embedded} if for some $n$ the
$A_i$ are homotopic (rel boundary) to pairwise disjoint embeddings
in the connected sum of $X$ with $n$ copies of $S^2\times S^2$.
\end{defn}
Here the sums are assumed to be taken along balls which are
disjoint from $A$. Note that in the literature `stable embedding'
usually means up to homology, rather than homotopy (see e.g.
\cite{B}).

The invariant $\tau_1(A)$ is an obstruction to homotoping $A$ to
an embedding, and the following corollary gives a stable converse:
\begin{cor}\label{cor:tau=stable embedding}
$\tau_1(A)$ vanishes if and only if $A$ can be stably embedded.
\end{cor}
Corollary~\ref{cor:tau=stable embedding} follows from applying the
``Norman trick'' (\cite{No}) to the height 1 Whitney tower in
statement (iii) of Theorem~\ref{thm:tau}, as explained in
subsection~\ref{subsec:proof-of-cor-tau=stableembedding} below.

\subsection{Stably slice links}\label{subsec:intro-stable-slice-links}

A link $L$ in a 3-manifold $M$ is \emph{stably slice} if the
components of $L\subset M\times\{0\}$ bound a collection of
properly immersed disks in $M\times I$ which can be stably
embedded.

As discussed later in
subsection~\ref{sub-sec:proof-of-thm-tau(L)}, for a link $L$ of
null-homotopic knots in $M$ the concordance invariant
$\tau_0(L):=\tau_0(D)\in\cT_0(\pi_1M)$, where $D$ is any
collection of properly immersed disks in $M\times I$ bounded by
$L\subset M\times\{0\}$, is the complete obstruction to $L$
bounding an order 1 Whitney tower on immersed disks in $M\times
I$. For links with vanishing order $0$ obstruction, there is a
similarly defined order $1$ invariant:
\begin{thm}\label{thm:tau(L)}
For $L\subset M$ with vanishing $\tau_0(L)$, the \emph{order 1
intersection tree} $\tau_1(L):=\tau_1(D)\in\cT_1(\pi_1M)$ is a
well-defined concordance invariant of $L$.
\end{thm}
Note that in this setting the intersection relations $INT(D)$ turn
out to be trivial so $\tau_1(L)$ takes values in $\cT_1(\pi_1M)$.
We get the following characterization of stably slice links.
\begin{cor}\label{cor:stable-slice}
$L$ is stably slice if and only if $\tau_1(L)=0\in\cT_1(\pi_1M)$.
\end{cor}
If $\pi_1M$ is non-trivial and left-orderable, as many 3-manifold
groups are \cite{BRW}, then $\cT_1(\pi_1M)$ is isomorphic to
$\Z^\infty\oplus \Z_2^\infty$
(Proposition~\ref{prop:orderable-T0-normal-form} below). Each
element of $\cT_1(\pi_1M)$ can be realized by a link formed by
tying copies of the Borromean rings into a trivial link
(Figure~\ref{BoroRings-and-PullApart-fig}).
\begin{figure}
\centerline{\includegraphics[width=135mm]{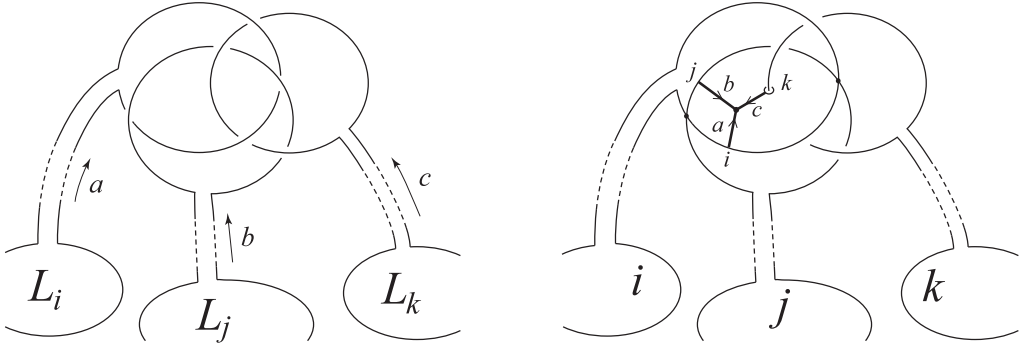}}
         \caption{Tying the Borromean rings into an unlink along bands
         representing elements $a$, $b$, and $c$ in $\pi_1M$
         yields a link $L$, with $\tau_1(L)$ represented by a
         Y-tree $t_p$ as in
         Figure~\ref{Y-wdisk-labelled-with-tree}. This is
         illustrated on the right, which shows $t_p$ in a Whitney
         disk for the cancelling pair of intersection points in a null
         homotopy of $L$.
         }\label{BoroRings-and-PullApart-fig}
\end{figure}

\begin{rem}[Finite type invariants] The reader familiar with finite
type invariants will recognize that operation of
Figure~\ref{BoroRings-and-PullApart-fig} can be effected by a
Y-graph clasper or clover surgery (\cite{H,GGP}). The links bounding
order 1 Whitney towers are exactly the $\pi$-\emph{algebraically
split} links which arise in the setting of finite type 3-manifold
invariants proposed by Garoufalidis and Levine in \cite{GL}. The
equivariant triple $\overline{\mu}$-invariants defined in \cite{GL},
which characterize \emph{surgery equivalence} of such links (in the
sense of \cite{Le}), correspond to the projection of $\tau_1(L)$ to
the $\Z^\infty$ factors of $\cT_1(\pi_1M)$. Thus the $\Z_2$ factors
of $\tau_1$ detect infinite families of surgery equivalent links
which are not (even stably) concordant.
\end{rem}

\begin{rem}[The Arf invariant] For a knot $K$ in a simply connected $M$,
$\tau_1(K)\in\cT_1(\pi_1M)\cong\Z_2$ is just the Arf invariant $\arf
(K)$, which is known to be the complete obstruction to stably
slicing $K$ (\cite{FK,Ma}). For $K$ null-homotopic in general $M$,
mapping $\pi_1M$ to the trivial group induces
$\cT_1(\pi_1M)\rightarrow\Z_2$ which takes $\tau_1(K)$ to $\arf
(K)$. It is easy to construct $K$ with trivial $\arf (K)$ but
non-trivial $\tau_1(K)$ by tying an even number of Borromean rings
(with varying $\pi_1M$ decorations) into the unknot. Similar
comments apply for the mod 2 reduction of the Sato-Levine invariant
of a 2-component link (\cite{Sato}).
\end{rem}

\begin{rem}
Since $\tau_1(-L)=-\tau_1(L)$, where $-L$ denotes $L$ with
reversed orientations on all components (fixing the orientation of
$M$), the $\Z^\infty$ factors in $\tau_1(L)$ represent
obstructions to the existence of a (stable) concordance between
$L$ and $-L$.
\end{rem}

\subsection{Stable
concordance of null-homotopic
links}\label{subsec:intro-stable-concordance}

Two links $L$ and $L'$ in $M$ are \emph{stably concordant} if
there exists a collection $A$ of properly immersed annuli in
$M\times I$, bounded by $L\subset M\times \{0\}$ and $L'\subset
M\times\{1\}$, such that $A$ can be stably embedded.

To classify stable concordance we will need to extend the
definition of $\tau_1$ to properly immersed \emph{annuli}
supporting an order 1 Whitney tower in such a way that
Theorem~\ref{thm:tau} and Corollary~\ref{cor:tau=stable embedding}
still hold. Restricting our attention to links of null-homotopic
knots corresponds to inessential annuli, that is, proper
immersions of annuli $(A,\partial A)\looparrowright(X,\partial X)$
which take the circle factors of $A$ to the trivial element of
$\pi_1X$. In this case, the new indeterminacies present in
$\tau_1(A)$ correspond to the boundaries of the Whitney disks
winding around the annulus, and are computable from $\partial
A\subset\partial X$. In the setting of links in 3-manifolds, these
new INT indeterminacies, and others coming from non-separating
2-spheres, turn out to only depend on the lower order linking
invariant $\tau_0(L)$, where the link $L\subset
\partial A$ consists of the half of the boundary components of $A$ sitting in
either end of $M\times I$.

For links of null-homotopic knots, order 1 Whitney concordance is
classified by $\tau_0(L)\in\cT_0(\pi_1M)$
(\ref{subsec:order1wconc} below and \cite{S1}). Any element
$z\in\cT_0(\pi_1M)$ can be realized as $\tau_0(L_z)$ for a link
$L_z$ created by adding clasps to the unlink, with the clasps
guided by loops determining the group elements in a representative
of $z$ (See Figure~\ref{clasplink-guidingarcs-fig}).
\begin{figure}
\centerline{\includegraphics[width=100mm]{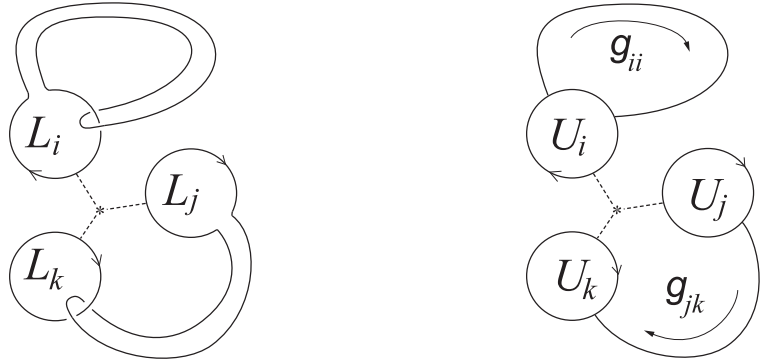}}
         \caption{On the left, a clasp link, built from the unlink on the right,
         using the indicated guiding arcs.}
         \label{clasplink-guidingarcs-fig}
\end{figure}

Fixing such $L_z$, for any $L$ with $\tau_0(L)=z$ we define the
\emph{relative order 1 intersection tree} $\tau_1(L_z,L)$:
$$
\tau_1(L_z,L):=\tau_1(A)\in\frac{\cT_1(\pi_1M)/\mbox{INT($z$)}}{\Phi(z)}
$$
where $A$ is any singular concordance between $L_z$ and $L$
admitting an order 1 Whitney tower. Here, the intersection
relations $\mbox{INT($z$)}:=\mbox{INT($A$)}$ turn out to only
depend on $z$; and the \emph{indeterminacy subgroup}
$\Phi(z)<\cT_1(\pi_1M)/\mbox{INT($z$)}$, which is determined by
order 1 intersections between 2-spheres and $L_z$, also turns out
to only depend on $z$ (Lemma~\ref{lem:link-sphere-pairing}, and
\ref{subsec:Phi-def}).

In irreducible 3-manifolds $\Phi(z)$ is trivial for all $z$, and in
many 3-manifolds the quotient by $\Phi(z)$ is sufficient to account
for all indeterminacies in the choice of $A$, so that
$\tau_1(L_z,L)$ is a well-defined link concordance invariant for all
$z$. However, as explained below in subsection~\ref{subsec:twisted
stabilizers}, slightly restricting $z$ will greatly expand the
number of cases for which $\tau_1(L_z,L)$ is well-defined. For now,
we observe that there is an action of the $m$-fold cartesian product
$\Psi :=(\pi_1M)^m$ on $\cT_0(\pi_1M)$ which corresponds to choices
of basings of the $m$ components of $L$ in computing $\tau_0(L)$.
The stabilizer $\Psi_z<\Psi$ of $z$ is called \emph{untwisted} if
all elements of $\Psi_z$ induce the trivial permutation of the
elements of $\pi_1M$ in a representative of $z$, and \emph{twisted}
otherwise (see \ref{subsec:Psi-stabilizers} below). This notion of
twisting will be seen to correspond to M\"{o}bius bands and singular
Seifert fibers, and in general a twisted stabilizer $\Psi_z$ can
only occur for certain special forms of $z$.

As explained in Section~\ref{sec:tau1(L,Lz)-proof}, by choosing a
clasp link $L_z$ in each order 1 Whitney concordance class we get
the following classification theorem for many links of
null-homotopic knots:
\begin{thm}\label{thm:tau(Lz,L)}
 For any orientable 3-manifold $M$ with $\pi_1M$ torsion-free,
and any $z\in \cT_0(\pi_1M)$ with $\Psi_z$ untwisted, there exists
a link $L_z$ in $M$ such that for any links $L$ and $L'$ with
$\tau_0(L)=\tau_0(L')=z\in\cT_0(\pi_1M)$ the following are
equivalent:
\begin{enumerate}
  \item $\tau_1(L_z,L)=\tau_1(L_z,L')$.
  \item $L$ and $L'$ are stably concordant.
  \item $L$ and $L'$ are order 2 Whitney concordant.
  \item $L$ and $L'$ are height 1 Whitney concordant.
\end{enumerate}
\end{thm}
Here the equality in statement (i) is understood to be up to an
action of $\Psi_z$ by left multiplication. By
Proposition~\ref{prop:twisted-stabilizers}, $\Psi_z$ will be
untwisted for all $z$ if $M$ contains no circle bundles over
non-orientable surfaces with orientable total space, and no
Seifert fibered submanifolds with singular fibers. The notions of
order $n$ (resp. height $n$) Whitney concordance refer to the
equivalence relations on links of cobounding immersed annuli
supporting an order $n$ (resp. height $n$) Whitney tower in
$M\times I$ (c.f. Theorem~\ref{thm:tau}).

It follows that $\tau_1(L_z,\cdot)$ is a concordance invariant,
and the effect of changing the choice of $L_z$ is computed in
Section~\ref{sec:examples}. Again, all elements in the target can
be realized by tying Borromean rings into $L_z$, but the INT($z$)
and $\Phi(z)$ relations show that many such Borromean surgeries do
not change the stable concordance class of $L_z$.

The presence of twisted stabilizers can lead to additional
indeterminacies coming from self-homotopies of $L_z$ as
illustrated below in
subsection~\ref{subsec:knots-in-B-twisted-S1}. A general method
for dealing with such additional indeterminacies would be
interesting, as would be an extension of $\tau_1$ and
Theorem~\ref{thm:tau(Lz,L)} to the case where $\pi_1M$ has
torsion.

\subsection{Stable
concordance of essential
knots}\label{subsec:intro-essential-knots} In the case where
$M\cong F\times S^1$ is the product of an orientable surface
$F\neq S^2$ with the circle, the stable concordance classification
of knots is completed by Theorem~\ref{thm:essential-tau(Lz,L)} at
the end of Section~\ref{sec:essential-knots}, which extends the
results of Theorem~\ref{thm:tau(Lz,L)} to knots in arbitrary
non-trivial free homotopy classes. The corresponding extension of
$\tau_1$ to essential annuli requires two changes: To account for
choices of paths, the decorations on Y-trees are taken in coset
spaces of $\pi_1M$ by a cyclic group; and to account for choices
of Whitney disk boundaries, the INT relations must be again
modified. The resulting INT relations turn out to correspond to
solutions of certain Baumslag-Solitar equations -- which are
highly restricted in 3-manifold groups -- and can be computed in
terms of lower order intersections among annuli, as explained in
Section~\ref{sec:essential-knots}.


\section{Intersection trees and stable embeddings}\label{sec:int-invariants}
This section recalls the definitions of the intersection
invariants $\tau_0$ and $\tau_1$ for simply connected surfaces
immersed in 4-manifolds, and sketches the proof that $\tau_1$ is a
well-defined homotopy invariant -- emphasizing aspects that will
be relevant to the generalization of $\tau_1$ to immersed annuli.
Proofs of Corollary~\ref{cor:tau=stable embedding} and
Theorem~\ref{thm:tau(L)} from the introduction are given in
subsections \ref{subsec:proof-of-cor-tau=stableembedding} and
\ref{sub-sec:proof-of-thm-tau(L)}. Additional background material
on immersed surfaces in 4-manifolds can be found in Chapter 1 of
\cite{FQ} (and absorbed from the examples below in
Section~\ref{sec:examples}).

\subsection{The order zero intersection tree $\tau_0$}\label{subsec:tau0}
For $A=A_1,A_2,\cdots,A_m$ a collection of properly immersed
simply connected oriented surfaces in an oriented 4-manifold $X$,
Wall's hermitian intersection form $\lambda$, $\mu$ is defined as
follows. The surfaces are perturbed into general position and
equipped with \emph{whiskers} (basepoints joined by a path to the
basepoint of $X$). To each transverse intersection point $p\in
A_i\cap A_j$ is associated a fundamental group element
$g_p\in\pi_1X$ determined by a loop through $A_i$ and $A_j$ which
changes sheets at $p$. Summing over all such intersection points,
with the usual notion of the sign $\epsilon_p=\pm 1$, defines the
intersection and self-intersection ``numbers'':
$$
\lambda(A_i,A_j)=\sum_{p\in A_i\cap A_j}\epsilon_p\cdot
g_p\in\Z[\pi_1X]
$$
and
$$
\mu(A_i)=\sum_{p\in A_i\cap A_i}\epsilon_p\cdot
g_p\in\frac{\Z[\pi_1X]}{\Z[1]\oplus\{z-\overline{z}\}}.
$$
Modding out by the action of the involution $z\mapsto
\overline{z}$ on $\Z[\pi_1X]$ induced by the map $g\mapsto g^{-1}$
on $\pi_1X$ accounts for choices of orientations on the
sheet-changing loops through the self-intersections of $A_i$. The
relations killing $\Z[1]$ correspond to indeterminacies in the
$\mu(A_i)$ due to local cusp homotopies which create
self-intersections in $A_i$ whose double-point loops determine the
trivial element $1\in\pi_1X$.

Basic singularity theory shows that $\lambda(A_i,A_j)$ and
$\mu(A_i)$ are invariant under homotopy of $A$ (rel boundary).
Since the indeterminacy relations can be realized (by introducing
local cusp homotopies and orienting double-point loops
appropriately), the vanishing of $\lambda$ and $\mu$ implies that,
after perhaps a homotopy (rel boundary), all intersection points
of $A$ occur in \emph{cancelling pairs} having equal group
elements and opposite signs. As described in detail below, each
such cancelling pair admits a Whitney disk.

If Whitney disks can be found that are all disjoint, with embedded
interiors disjoint from $A$, and satisfying a normal framing
condition described below, then such Whitney disks can be used to
guide a homotopy of $A$ to an embedding. It can always be arranged
that the Whitney disks are disjoint and framed, but in general
they will have interior intersections with $A$. The obstruction
theory of Whitney towers attempts to pair up these higher order
intersections with higher order Whitney disks, and the resulting
invariants associated to such a tower of Whitney disks naturally
take values in groups generated by unitrivalent trees. These trees
are associated to unpaired intersection points, and sit as subsets
bifurcating down through the Whitney tower, with the trivalent
vertices corresponding to Whitney disks, and the edges
corresponding to sheet-changing paths.

In this language, the proper immersion $A$ is called a
\emph{Whitney tower of order zero}, and Wall's intersection form
is an invariant of order zero (since there are no Whitney disks
involved). In the notation of Whitney towers, the information
contained in $\lambda$ and $\mu$ is expressed as the order zero
\emph{intersection tree}
$$
\tau_0(A):=\sum \epsilon_p\cdot t_p\in\cT_0(\pi_1X)
$$
where each $t_p$ is an oriented edge (order zero trees have no
trivalent vertices) which is decorated by $g_p$ and has endpoint
vertices labelled by $i$ and $j$ for $p\in A_i\cap A_j$. We will use
the notation $(g)_{ij}$ to denote an order zero tree which is
decorated by $g$ and oriented from $i$ to $j$. The target
$\cT_0(\pi_1X)$ is the abelian group additively generated by such
decorated edges modulo two kinds of relations: The orientation
relations $(g)_{ij}=(g^{-1})_{ji}$ reflect the Hermitian nature of
the pairing $\lambda$ for $i\neq j$, and correspond for $i=j$ to the
relations $z=\overline{z}$ in the self-intersection invariants $\mu
(A_i)$ of the components of $A$. The framing relations $(1)_{ii}=0$
correspond to the relations $\Z[1]=0$ in the $\mu (A_i)$ by killing
any edges decorated by trivial group elements and having both
vertices labeled by $i$. (With these framing relations one has a
homotopy invariant; without the framing relations one has a
\emph{regular} homotopy invariant.)

Thus, $\tau_0(A)$ is exactly Wall's intersection form rewritten in
Whitney tower notation, and the vanishing of $\tau_0(A)$ is
equivalent to the existence of Whitney disks pairing all
intersection points in $A$ (perhaps after performing some cusp
homotopies). In the language of Whitney towers we say that such
$A$ admits an \emph{order 1 Whitney tower}. The Whitney disks in
an order one Whitney tower are required to have pairwise
disjointly embedded boundaries. The interiors of the Whitney disks
are allowed to be immersed, and may intersect each other as well
as $A$. All Whitney disks are required to be properly framed as
described next.

\subsection{Whitney disk framings.}\label{subsec:w-disk-framings}
Consider a cancelling pair of intersections $p$ and $q$ between
$A_i$ and $A_j$, with $g_p=g_q\in\pi_1X$ and
$\epsilon_p=-\epsilon_q$. The union of any pair of arcs, one in
$A_i$ and the other in $A_j$, each connecting $p$ and $q$ (but
avoiding all other singularities), forms a loop contractible in
$X$, and any (immersed) disk $W$ bounded by such a loop is a
\emph{Whitney disk}. Since $p$ and $q$ have opposite signs, the
restriction $\nu_{\partial W}$ of the normal disk bundle $\nu_W$
of $W$ in $X$ to $\partial W$ is the trivial bundle $S^1\times
D^2$ (if the signs of $p$ and $q$ were equal, then we would get
the non-orientable disk bundle over the circle).

\begin{figure}
\centerline{\includegraphics[width=135mm]{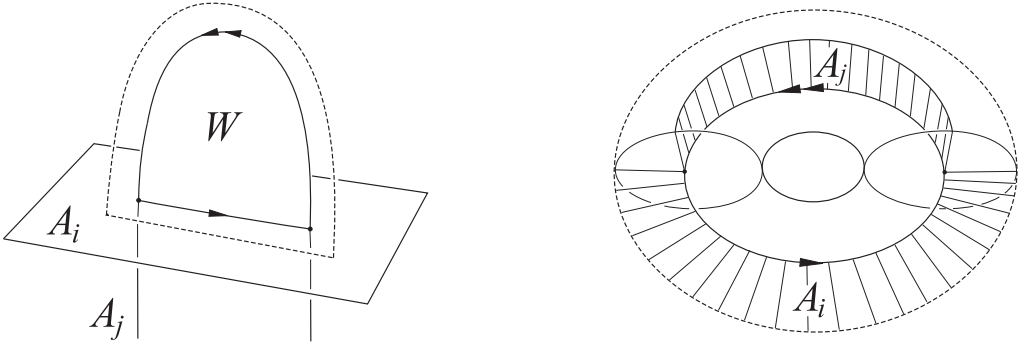}}
         \caption{The Whitney section over the boundary of a (zero-)framed Whitney disk is
         indicated by the dotted loop shown on the left for a clean Whitney disk $W$ in
         a 3-dimensional slice of 4-space, and on the right in the (pulled back) normal
         disk bundle over $\partial W$.}
         \label{Framing-of-Wdisk-fig}
\end{figure}
The trivial $D^2$-bundle $\nu_{\partial W}$ has a
nowhere-vanishing \emph{Whitney section} constructed as follows
(see Figure~\ref{Framing-of-Wdisk-fig}): Denote by $\partial_i W$
(resp. $\partial_j W$) the arc of $\partial W$ that lies in $A_i$
(resp. $A_j$). Over $\partial_i W$ choose a splitting of
$\nu_{\partial W}$ given by a vector field $v_i^t$ tangent to
$A_i$ and a vector field $v_i^n$ normal to both $W$ and $A_i$.
Over $\partial_j W$ we have a similar splitting $\{ v_j^t, v_j^n
\}$ of $\nu_{\partial W}$, and we may arrange that these
splittings coincide at $p$ and $q$ with $v_i^t=v_j^n$ and
$v_i^n=v_j^t$. The Whitney section of $\nu_{\partial W}$ is
constructed by taking $v_i^t$ over $\partial_i W$, and $v_j^n$
over $\partial_j W$. If this section of $\nu_{\partial W}$ can be
extended to a nowhere-vanishing section of $\nu_W$, then we say
that $W$ is \emph{framed}.

For a chosen orientation, the relative Euler number of the Whitney
section gives an integer obstruction to extending the section over
$W$, so ``framed'' really means ``zero-framed''. This obstruction
can always be killed at the cost of creating interior
intersections between $W$ and $A_i$ (or $A_j$) by a \emph{boundary
twisting} modification of $W$ which changes the obstruction by
$\pm 1$. Also, the framing obstruction can be changed by $\pm 2$
at the cost of creating an interior self-intersection of $W$ by
performing a cusp homotopy, called an \emph{interior twist}.
Boundary and interior twists are illustrated below in
Figure~\ref{clasp-link-boundary-twist-grey-fig} and
Figure~\ref{InteriorTwistPositiveEqualsNegative-fig} of
Section~\ref{sec:examples}.

\subsection{The order 1 intersection tree $\tau_1$}\label{subsec:tau1-def}
The invariant denoted $\tau$ in \cite{ST1} corresponds in the
general theory of Whitney towers (\cite{CST,CST2,ST2}) to the
order 1 intersection tree $\tau_1$. For a collection $A$ of
properly immersed simply connected surfaces with vanishing
$\tau_0(A)$, $\tau_1(A)$ is defined as follows. Choose an order 1
Whitney tower $\cW$ on $A$, with fixed whiskers on the $A_i$ and
on each of the Whitney disks in $\cW$. To each intersection point
$p\in W_{(i,j)}\cap A_k$ between $A_k$ and the interior of a
Whitney disk $W_{(i,j)}$ pairing intersections between $A_i$ and
$A_j$ is associated a trivalent $Y$-tree $t_p$ as illustrated in
Figure~\ref{Y-wdisk-labelled-with-tree}. The univalent vertices
are labelled by $i$, $j$ and $k$, and the oriented edges are
decorated by elements $a$, $b$, and $c$ in $\pi_1X$ determined by
sheet-changing loops through $W_{(i,j)}$ and the components of $A$
using the chosen whiskers, with the loop orientations
corresponding to the edge orientations. Fixing (arbitrary)
orientations on the Whitney disks in $\cW$ associates a sign
$\epsilon_p=\pm 1$ to each interior intersection point $p$, and
determines a cyclic orientation at the trivalent vertex of the
corresponding tree $t_p$ via the following convention: As
illustrated in Figure~\ref{Y-wdisk-labelled-with-tree}, the tree
$t_p$ sits as an embedded subset of $W_{(i,j)}$. The two edges of
$t_p$ labelled $i$ and $j$ determine a ``corner'' of $W_{(i,j)}$
which does not contain the $k$-labelled edge, and the orientation
of $t_p$ is taken to be that induced by $W_{(i,j)}$ if and only if
this corner contains the \emph{positive} intersection point
between $A_i$ and $A_j$ that is paired by $W_{(i,j)}$.
\begin{figure}
\centerline{\includegraphics[width=110mm]{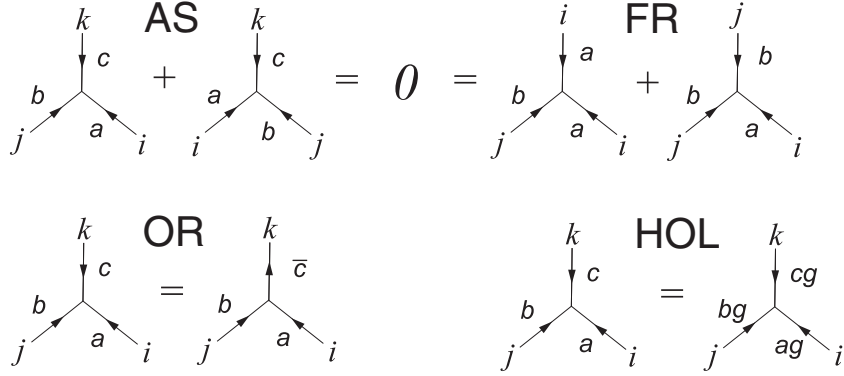}}
         \caption{The AS \emph{antisymmetry}, FR \emph{framing}, OR \emph{orientation}, and HOL \emph{holonomy}
         relations in $\cT_1(\pi_1X)$, assuming a fixed planar orientation on the trivalent vertex
         of all Y-trees. Here $a$, $b$, $c$, and $g$ are elements of $\pi_1X$, with
         $\overline{c}=c^{-1}$.
         These edge labels can also be extended linearly to elements of $\Z[\pi_1X]$.}
         \label{Y-relations-fig}
\end{figure}

(Due to differing orientation conventions there is a global sign
difference between $\tau$ of \cite{ST1} and the $\tau_n$ of more
recent papers (\cite{CST,ST2}), including $\tau_1$ here.)

The \emph{order 1 intersection tree} $\tau_1(\cW)$ is defined by
summing the $t_p$ over all unpaired intersection points in $\cW$:
$$
\tau_1(\cW):=\sum \epsilon_p\cdot
t_p\in\frac{\cT_1(\pi_1X)}{INT(A)}
$$

The abelian group $\cT_1(\pi_1X)$ is additively generated by the
above described decorated $Y$-trees, modulo the relations shown in
Figure~\ref{Y-relations-fig}. The INT($A$) relations are shown in
Figure~\ref{INTrelationTubed2sphere-fig}.
\begin{figure}
\centerline{\includegraphics[width=125mm]{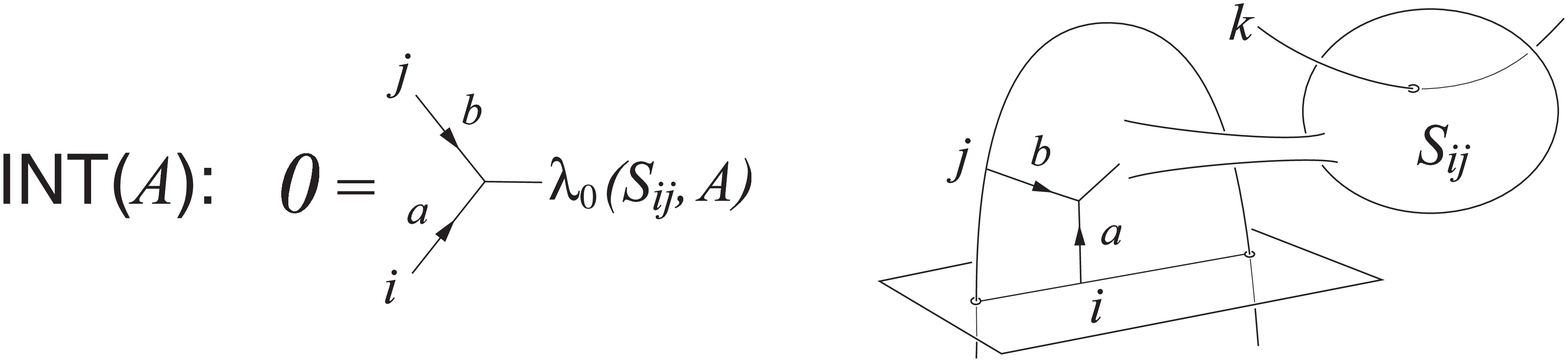}}
         \caption{ The INT(A) \emph{intersection} relations.
         Here $S_{ij}$ varies over 2-spheres representing a generating set for $\pi_2X$,
         and the order zero invariant $\lambda_0(S_{ij},A)$ counts only intersections between $S_{ij}$ and $A$,
         with the vertex label $k$ assigned to the intersections between $S_{ij}$ and $A_k$
         (and edge labels from $\pi_1X$).}
         \label{INTrelationTubed2sphere-fig}
\end{figure}

\subsubsection{Notation and normal forms for $\cT_1$}\label{subsubsec:T1notation}
Note that the AS relation implies that the generators appearing in
the FR relation are 2-torsion. By the OR edge-orientation
relation, we can normalize all edges on generators of
$\cT_1(\pi_1X)$ to be oriented towards the trivalent vertex. We
will use the cyclically oriented triple $(a,b,c)_{ijk}$, to denote
such a Y-tree having univalent labels $i$, $j$, $k$, and
corresponding group elements $a$, $b$, and $c$.

For fixed triples of univalent vertex labels, it will be
convenient to represent the corresponding subgroups of
$\cT_1(\pi_1X)$ as quotients of the integral group ring
$\Z[\pi_1X\times \pi_1X]\cong Z[\pi_1X]\times Z[\pi_1X]$ by using
the HOL relation to normalize one edge decoration to the trivial
group element $1\in\pi_1X$ -- for instance
$(a,b,c)_{ijk}=(1,ba^{-1},ca^{-1})_{ijk}$ -- and summing group
element decorations inside the parentheses. This corresponds to
the notation used in \cite{ST1} for $\tau =\tau_1$.

Assume that $\prec$ is a \emph{right-order} on $\pi_1X$ -- a
strict total ordering of the elements of $\pi_1X$ such that if
$g\prec h$ then $gf\prec hf$ -- and that the components of $A$
(the univalent vertex labels) are ordered. Then we have the
following normal form for $\cT_1(\pi_1X)$:

First consider a Y-tree whose univalent vertices all have the same
label, and whose edge decorations are distinct group elements. Use
the HOL relation to trivialize one edge decoration, and denote the
resulting Y-tree by the ordered pair $(g,h):=(1,g,h)$, where we
suppress the common univalent vertex label.

Such a generator $(g,h)$ is not involved in any FR relation (since
$1\neq g\neq h\neq 1$), and the orbit of $(g,h)$ under the AS and
HOL relations in this notation is:
$$
(g,h)=-(g^{-1},hg^{-1})=(hg^{-1},g^{-1})=-(gh^{-1},h^{-1})=(h^{-1},gh^{-1})=-(h,g).
$$
Now one can check that for each of the possible relative orderings
of $1$, $g$, and $h$ given by $\prec$, there is exactly one
generator $(a,b)$ in the orbit of $(g,h)$ such that $1\prec a\prec
b$; for instance if $h\prec 1\prec g$, then $(h^{-1},gh^{-1})$
satisfies $1\prec h^{-1}\prec gh^{-1}$.

In the case where $g=h$, but still assuming all univalent labels
are the same, then the generator is 2-torsion and we have the
following orbit under the AS, HOL, and FR relations:
$$
(g,g)=(g^{-1},g^{-1})=(g^{-1},1)=(1,g^{-1})=(g,1)=(1,g).
$$
From each such orbit we get a unique generator $(1,a)$ such that
$1\preceq a$ (or a unique generator $(a,a)$ with $1\preceq a$).

This same approach works for other combinations of univalent
labels, of course with different possible choices of convention.
We state here the rest of a set of generators in normal form,
written in ordered triple notation, with non-decreasing univalent
labels (from left to right) which correspond to the cyclic
orientation of the Y-tree:
$$
(1,g,h)_{ijk}\quad\quad\mbox{for distinct labels}\quad i < j < k,
$$
$$
(1,g,h)_{iij} \quad \mbox{with}\quad 1 \prec g \quad\quad
\mbox{for labels} \quad i < j,
$$
$$
(1,g,h)_{ijj} \quad \mbox{with} \quad g \prec h \quad\quad
\mbox{for labels} \quad i < j,
$$
and
$$
(1,1,g)_{iij} \quad \quad \mbox{for labels} \quad i < j.
$$
These last generators are 2-torsion.

This construction of a normal form, together with the observation
that a (non-trivial) right-orderable group is infinite (it must be
torsion free), gives the following proposition:
\begin{prop}\label{prop:orderable-T0-normal-form}
If $\pi_1X$ is non-trivial and right-orderable, then
$\cT_1(\pi_1X)$ is isomorphic to $\Z^\infty \oplus \Z_2^\infty$.
\end{prop}

\subsection{Homotopy invariance of
$\tau_1$}\label{subsec:htpy-invariance-of-tau-proof} We summarize
here the proof in \cite{ST1} that $\tau_1(A):=\tau_1(\cW)$ is a
well-defined (regular) homotopy invariant of $A$, indicating the
relevant points that will have to be addressed when we allow $A$
to have annular components. The proof proceeds by first checking
that the value of $\tau_1(\cW)$ in
${\cT_1(\pi_1X)}/{\mbox{INT}(A)}$ does not depend on the choices
in constructing the Whitney tower $\cW$. Showing homotopy
invariance is then reduced to checking invariance under finger
moves.

With an eye towards our interest in the case $X=M\times I$, we
will make the simplifying assumptions that $\pi_1X$ is
torsion-free, and that $X$ is spin, i.e., the second
Stiefel-Whitney class $\omega_2 X$ is trivial.

\textbf{Independence of tree decorations and signs:} First of all,
the group elements decorating the tree edges do not depend on the
choices of sheet-changing loops since the $A_i$ and the Whitney
disks are simply connected. The OR relations account for the
orientation choices on the loops, which we will always assume to
be oriented \emph{into} the Whitney disks, so that all tree edges
are oriented towards the trivalent vertex. The decorations do not
depend on the choices of whiskers on the Whitney disks by the HOL
relations.

The sign $\epsilon_p$ associated to each cyclically oriented $t_p$
is independent of the chosen Whitney disk orientations (for fixed
orientations on the $A_i$) by the AS antisymmetry relations and
the ``positive corner'' convention
(subsection~\ref{subsec:tau1-def}).

\textbf{Independence of Whitney disk interiors:} For fixed
boundaries on the Whitney disks, the INT and FR relations guarantee
that changing the Whitney disk interiors does not change
$\tau_1(A)$; the idea is that such a change is given by
intersections with 2-spheres: Consider any Whitney disk
$W=W_{(i,j)}$ in $\cW$ pairing $p$ and $q$ in $A_i\cap A_j$. If $W'$
is another Whitney disk with same boundary as $W$, then the union of
$W$ and $W'$ is a topological 2-sphere $S$, which may not be smooth
along $\partial W=\partial W'$. We may arrange (by a small isotopy
and after perhaps performing some boundary twists on $W$) that along
their common boundary the collars $C$ and $C'$ of $W$ and $W'$ point
in opposite directions (Figure~\ref{collar-equator-pushoff-fig}) so
that $W\cup W'$ is an immersed 2-sphere, still denoted $S$, which is
oriented by the orientation of $W$ together with the opposite
orientation of $W'$.
\begin{figure}
\centerline{\includegraphics[width=100mm]{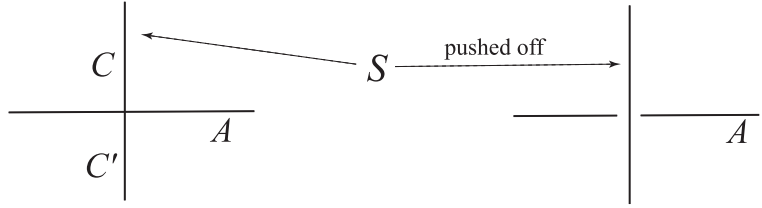}}
         \caption{The 2-sphere $S$ formed as the union of two Whitney disks
         can be pushed off of $A$ near its equator -- where the collars of the
         Whitney disks are joined along their boundaries.}
         \label{collar-equator-pushoff-fig}
\end{figure}
The boundary twists on $W$ ($n$ twists into $A_i$ and $m$ twists
into $A_j$) have changed $\tau_1(\cW)$ by
$n(a,b,a)_{iji}+m(a,b,b)_{ijj}\in\cT_1(\pi_1X)$ -- which is
2-torsion by the AS relations -- and have changed the framing of $W$
by $n+m\in\Z$. But by extending the Whitney section over the equator
to a normal push-off of all of $S$ we see that $n+m$ is equal modulo
$2$ to $\omega_2 (S)\in\Z_2$ which vanishes since $X$ is spin (note
that any self-intersections of $S$ do not contribute to $\omega_2
(S)$). So by the AS and FR relations
$n(a,b,a)_{iji}+m(a,b,b)_{ijj}=0\in\cT_1(\pi_1X)$.

Since the circle of intersection between $S$ and $A$ can be
perturbed away without creating any new intersections
(Figure~\ref{collar-equator-pushoff-fig}), it follows that the
change in $\tau_1(\cW)$ resulting from replacing $W$ by $W'$ is
exactly described by an INT($A$) relation with $S=S_{ij}$.

\textbf{Independence of Whitney disk boundaries:} To show that
$\tau_1(\cW)$ does not depend on the choices of boundaries of the
Whitney disks, for fixed pairings of the cancelling singularities
in $A$, it is convenient to weaken the definition of an order 1
Whitney tower by allowing transverse intersections among the
boundaries of the Whitney disks (following \cite{ST1} and 10.8 of
\cite{FQ}). The definition of $\tau_1$ is extended to such Whitney
towers by assigning trees to the boundary intersections between
Whitney disks in the following way (see
Figure~\ref{W-disk-boundary-int-fig}).
\begin{figure}
\centerline{\includegraphics[width=100mm]{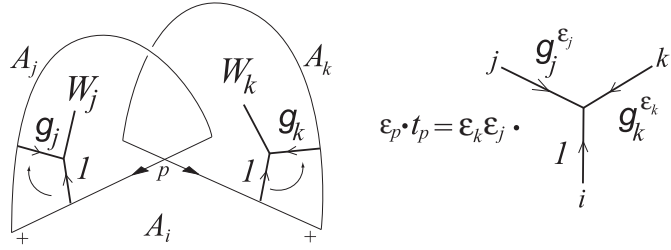}}
         \caption{The signed tree $\epsilon_p\cdot t_p$ assigned to an intersection point
         $p\in\partial_{\epsilon_j}W_{j}\cap\partial_{\epsilon_k}W_{k}$
         between Whitney disk boundaries in the case $\epsilon_k=+=\epsilon_j$,
         with the orientations of $A_i$ and $t_p$ both the same as the plane of the paper.}
         \label{W-disk-boundary-int-fig}
\end{figure}

Consider two Whitney disks, $W_{j}$ and $W_{k}$, pairing
intersections that $A_i$ has with $A_j$ and $A_k$, respectively.
Note that even if $W_{j}$ and $W_{k}$ have no interior intersections
with any $A_l$, they still have naturally associated Y-trees, with
only the ``descending'' edges and vertices decorated and labelled,
as in the left hand side of Figure~\ref{W-disk-boundary-int-fig}.
Choose whiskers so that the edges labelled by $i$ are decorated by
the trivial element $1\in\pi_1X$, which then determines elements
$g_j$ and $g_k$, respectively, decorating the edges labeled $j$ and
$k$. The cyclic orientation of a Y-tree sitting in a Whitney disk
$W$ corresponds to the orientation of $W$ by the ``positive corner''
convention described above (\ref{subsec:tau1-def}), and induces an
orientation of $\partial W$ (which corresponds to the convention
that $\overrightarrow{\partial W}$ together with a second inward
pointing vector give the orientation of $W$). We will use the
notation $\partial_+$ to indicate a Whitney disk boundary arc that
is oriented towards its positive intersection point, and
$\partial_-$ for a Whitney disk boundary arc that is oriented
towards its negative intersection point.

Now let $p\in
\partial_{\epsilon_j}W_{j}\cap\partial_{\epsilon_k}W_{k}$ be a point
such that the ordered pair of tangent vectors
$(\overrightarrow{\partial_{\epsilon_j}
W_{j}},\overrightarrow{\partial_{\epsilon_k} W_{k}})_p$ is equal
to the orientation of $A_i$ at $p$. Define the tree $t_p$
associated to such a $p$ by:
$$
\epsilon_p\cdot t_p:=\epsilon_k \epsilon_j
(1,g_k^{\epsilon_k},g_j^{\epsilon_j})_{ikj}
$$
where $\epsilon\in\{+,-\}\widehat{=}\{+1,-1\}$.

One can check that this definition of $t_p$ does not depend on the
choices made. The extended version of $\tau_1(A)$ is defined by
including such $t_p$ in the sum. Since all boundary intersections
can be eliminated by finger moves which create interior
intersections having the exact same trees (Figure~3 in
\cite{ST1}), this extended definition can always be reduced to the
original one.

Note that, properly interpreted, the formula assigning $t_p$ to
$p\in\partial W_j\cap\partial W_k$ also works when $W_j=W_k$,
including the case $i=j=k$. For instance, for $W$ pairing
self-intersections of $A_i$, and $p\in\partial_-W\cap\partial_+W$
such that the orientation of $A_i$ is equal to
$(\overrightarrow{\partial_{-} W},\overrightarrow{\partial_{+}
W})_p$, then $\epsilon_p\cdot t_p=-(1,g,g^{-1})_{iii}$, where the
Y-tree associated to $W$ has decorations $1$ and $g$ on the edges
dual to the $+$ and $-$ boundary arcs respectively. (This formula
will be used in Section~\ref{sec:examples} below.)

The proof of independence of Whitney disk boundaries now goes as
follows. Since the components of $A$ are all simply connected, any
configuration of Whitney disk boundaries can be achieved by a
regular homotopy of (collars of) the Whitney disk boundaries,
fixing the intersection points of $A$ (Clarification: we mean here
that this regular homotopy is induced by a regular homotopy of the
inverse images of the Whitney disk boundaries in the domain of
$A$, and extends to a regular homotopy of collars of the Whitney
disks in $X$). During such a homotopy, $\tau_1$ does not change
since boundary intersections come and go in cancelling pairs, or
accompanied by a cancelling interior intersection (when pushing
over an intersection point of $A$, see Figure~5 in \cite{ST1}).

Note that this step uses the fact that (the domains of) the
components of $A$ are simply connected, and will have to be
modified when we allow $A$ to have immersed annular components.

\textbf{Independence of intersection pairings:} The independence
of $\tau_1(\cW)$ on the choices of pairings of the intersections
of $A$ follows easily from a construction pictured in Figure~6 of
\cite{ST1} (originally from 10.8 of \cite{FQ}). In the presence of
2-torsion in $\pi_1X$ there is also a subtle indeterminacy
corresponding to the pairing of the inverse images of the
self-intersections of the $A_i$ which was first explained in
\cite{St}, and is covered by a more general INT relation
(\cite{ST1}) (which also covers the case where $X$ is not spin).

\textbf{Homotopy invariance:} Having established the independence
of $\tau_1(A):=\tau_1(\cW)$ on the choice of Whitney tower $\cW$,
homotopy invariance can be seen as follows. Up to isotopy, a
generic homotopy (rel $\partial$) between surfaces in a 4-manifold
is a sequence of finger moves (which create a pair of
intersections), Whitney moves (which eliminate a pair of
intersections), and local cusp homotopies (births and deaths of
local self-intersections). In a \emph{regular} homotopy it may be
arranged that there are \emph{only} finger moves and Whitney
moves, and that the finger moves all occur before the Whitney
moves. Since finger moves and Whitney moves are inverse to each
other, it follows that if $A$ is regularly homotopic to $A'$, then
there is $A''$ which differs (up to isotopy) from each of $A$ and
$A'$ by only finger moves which create cancelling intersections
paired by local clean Whitney disks. Since these finger moves and
their local clean Whitney disks can be assumed to be disjoint from
all other Whitney disks in any Whitney tower on $A$ or $A'$, it
follows that $\tau_1(A)=\tau_1(A'')=\tau_1(A')$.

\textbf{Geometric characterization:} That the vanishing of
$\tau_1(A)$ leads to an order 2 and height 1 Whitney tower follows
from the fact that all the relations in $\cT_1(\pi_1X)$ can be
realized by controlled manipulations of Whitney towers
(\cite{ST1}). For instance, the FR and INT relations can be
realized by creating clean Whitney disks, then twisting and tubing
into 2-spheres. The new relations introduced later for the
generalization of $\tau_1(A)$ which allows $A$ to have annular
components will be similarly realizable.

\subsection{Proof of Corollary~\ref{cor:tau=stable
embedding}}\label{subsec:proof-of-cor-tau=stableembedding}
\begin{proof}
The equivalence of statements (i) and (iii) in
Theorem~\ref{thm:tau} is proved as Theorem~2 of \cite{ST1} in the
case where $A$ consists of a single connected component, but the
exact same proof goes through for multiple components. Thus, the
vanishing of $\tau_1(A)$ means that all the singularities of $A$
can be paired by (framed) Whitney disks whose interiors are
disjoint from $A$. If the Whitney disk interiors are also
disjointly embedded, then $A$ is homotopic to an embedding without
any stabilization. Otherwise, it may be arranged by splitting the
Whitney disks using finger moves that each Whitney disk is
embedded, and has at most a single transverse interior
intersection with some other Whitney disk. Each such intersection
point $p\in W\cap W'$ can be eliminated after taking a connected
sum of $X$ with a copy of $S^2\times S^2$ by the Norman trick
(\cite{No}), which is illustrated schematically in
Figure~\ref{Norman-trick-fig}: If $S$ and $S'$ are dual 2-sphere
factors generating $\pi_2(S^2\times S^2)$, then $W$ can be tubed
into $S$, and $p$ can be eliminated by tubing $W'$ into $S'$ along
$W$. This operation can be done without creating any new
singularities, and since $S$ and $S'$ are both 0-framed, the
resulting Whitney disks are still framed. After one stabilization
for each interior intersection between Whitney disks, the
resulting disjointly embedded framed Whitney disks guide a
homotopy of $A$ to an embedding in the connected sum of $X$ with
(finitely many) copies of $S^2\times S^2$.
\begin{figure}
\centerline{\includegraphics[width=100mm]{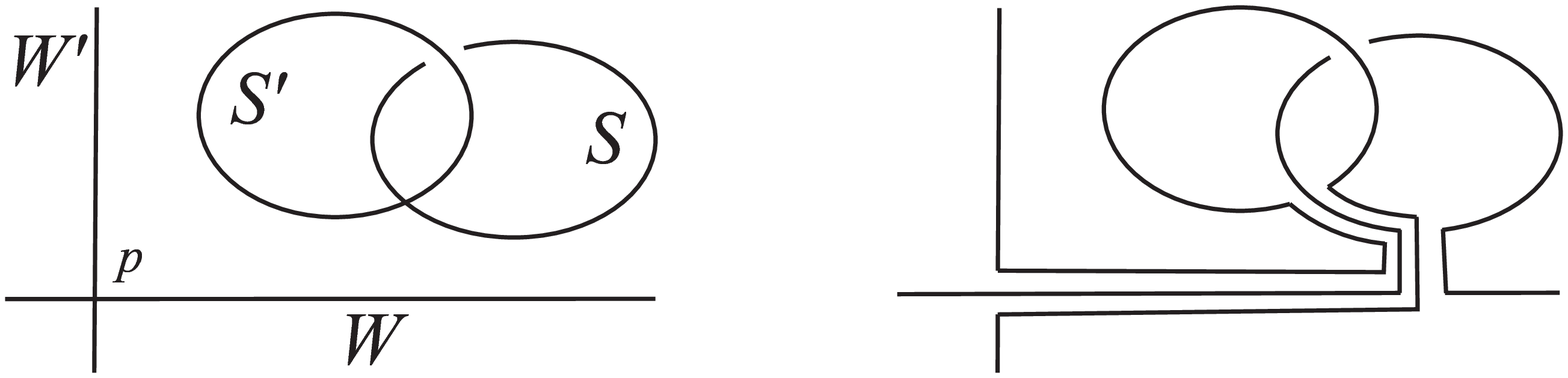}}
         \caption{The Norman trick.}
         \label{Norman-trick-fig}
\end{figure}

On the other hand, if $A\looparrowright X$ is homotopic to an
embedding in the connected sum of $X$ with copies of $S^2\times
S^2$, then since $\tau_1(A)$ only depends on the homotopy class
(rel boundary) of $A$, and since taking connected sums of $X$ with
$S^2\times S^2$ does not change the INT($A$) relations, it follows
that $\tau_1(A)$ vanishes.
\end{proof}


\subsection{Proof of Theorem~\ref{thm:tau(L)}}
\label{sub-sec:proof-of-thm-tau(L)} This subsection gives a proof
of the concordance invariance of the order 1 intersection tree
$\tau_1(L)$ of a link (Theorem~\ref{thm:tau(L)}), which at the
same time shows the concordance invariance of the order zero
intersection tree $\tau_0(L)$.

\begin{proof}
Consider an oriented $m$-component link $L$ of null-homotopic
knots in an oriented 3-manifold $M$, with $D$ any collection of
properly immersed disks in $M\times I$ bounded by $L\subset
M\times\{0\}$. For $n$ equal to $0$ or $1$, if $D$ admits an order
$n$ Whitney tower, then we want to show that
$\tau_n(L):=\tau_n(D)$ only depends on the concordance class of
$L$.

The key lemma is:
\begin{lem}[\cite{S}]\label{embedded-sphere-lemma}
Any $m$ elements of $\pi_2(M\times I)$ are represented by $m$
embedded pairwise disjoint 2--spheres.
\end{lem}

\begin{proof}
A well known consequence of the Sphere Theorem of 3-manifold
theory is that $\pi_2M$ is generated as a module over $\pi_1M$ by
disjoint embeddings (the 2--spheres that decompose $M$ into prime
factors, together with any spherical boundary components and
cross-sections of any $S^1 \times S^2$ factors, see Proposition
3.12 of \cite{Ha}). Tubing these generators together in $M\times
I$ does not create any new intersections, so all elements of
$\pi_2(M\times I)$ are represented by embeddings, which can be
isotoped to be pairwise disjoint using the product structure.
\end{proof}

We also need the following general properties of $\tau_n(A)$ for
$A\looparrowright X$, which can be checked directly from the
definitions: If $-A$ denotes a flip of orientation on (all
components of) $A$, then $\tau_0(-A)=\tau_0(A)$ and
$\tau_1(-A)=-\tau_1(A)$, where the orientation of the ambient
4-manifold $X$ is fixed. On the other hand, flipping the
orientation of $X$ while fixing the orientation of $A$ has the
effect of multiplying $\tau_0(A)$ by $-1$, but preserves
$\tau_1(A)$.

To see that $\tau_n(L)$ does not depend on the choice of $D$, let
$D'$ be another singular null-concordance of $L$. Then the union
$S$ of $D$ and $D'$ along $L$, in two copies of $M\times I$
identified along $M$, determines $m$ elements of $\pi_2(M\times
I)\cong \pi_2M$. The definition of $\tau_n(L):=\tau_n(D)$ depends
on fixing a convention for the how the orientation of $L\subset
M\times \{0\}$ induces an orientation of $D\looparrowright M\times
I$. It follows that, after orienting $S$ in $M\times I$ (which
requires reorienting one of the original copies of $M\times I$ and
the corresponding singular null-concordance), we have
$$
0=\tau_n(S)=\pm(\tau_n(D)-\tau_n(D'))
$$
where the first equality comes from
Lemma~\ref{embedded-sphere-lemma} and the homotopy invariance of
$\tau_n$.

It also follows from Lemma~\ref{embedded-sphere-lemma} that the
INT($D$) relations are trivial in the case $n=1$.

The sense in which $\tau_n(L)$ is a concordance invariant requires
interpreting its value modulo the effect on $\cT_n(\pi_1M)$ of
whisker choices. This will be discussed in detail in
Section~\ref{sec:tau1(L,Lz)-proof}. For now we observe that if $A$
is a concordance from $L'$ to $L$ then, up to the
change-of-whisker action, $\tau_n(L')=\tau_n(A\cup_L
D)=\tau_n(D)=\tau_n(L)$ since $A$ has no singularities. This same
argument shows that $\tau_n(L)$ is invariant under order $n+1$
Whitney concordance since then all singularities of $A$ of order
less than $n+1$ occur in cancelling pairs.
\end{proof}


\section{$\tau_1$ for inessential
annuli}\label{sec:tau-for-inessential-annuli}
In this section we
extend the definition of $\tau_1(A)$ to allow $A$ to have
components which are \emph{inessential} properly immersed annuli,
meaning that the induced map $\Z \rightarrow \pi_1X$ on
fundamental groups is trivial.

This generalized $\tau_1$ is illustrated by a pairing between
links and 2-spheres in $M$ described in
\ref{subsec:link-sphere-pairing} below, and computed in
\ref{subsec:computing-link-sphere-pairing}. These results are used
in \ref{subsec:Phi-def} to define the indeterminacy subgroup
$\Phi(z)$ in the definition of the relative order 1 intersection
tree $\tau_1(L_z,L)$ of \ref{subsec:intro-stable-concordance}.
(The reader who is only interested in links in irreducible
3-manifolds can skip these three subsections.)

We continue to assume that $X$ is spin, and $\pi_1X$ is
torsion-free.

Let $A\looparrowright X$ be a collection of properly immersed
surfaces admitting an order 1 Whitney tower. We now allow the
components of $A$ to include inessential annuli as well as 2-disks
and 2-spheres. In order to extend the definition
(\ref{subsec:tau1-def}) of $\tau_1(A)$ to such $A$ it will only be
necessary to add some additional INT($A$) relations.

Starting with the original definition \ref{subsec:tau1-def}, we
proceed by examining the steps in the proof in
\ref{subsec:htpy-invariance-of-tau-proof} that $\tau_1(A)$ is a
well-defined homotopy invariant.

First of all, note that $\tau_0(A)$ is still a well-defined
homotopy invariant for such $A$ (with the vanishing of $\tau_0(A)$
equivalent to the existence of an order 1 Whitney tower on $A$)
since in inessential annuli the elements of $\pi_1X$ associated to
sheet-changing paths are still well-defined (because loops in $A$
which do not change sheets must represent $1\in\pi_1X$). For the
same reason the edge decorations from $\pi_1X$ on the trees in
$\tau_1(A)$ are still well-defined.

The key step in \ref{subsec:htpy-invariance-of-tau-proof} which
relies on the assumption that the components of $A$ are simply
connected is the argument showing independence of the choices of
Whitney disk boundaries, which uses the fact that any two
configurations of arcs with fixed endpoints in a simply connected
surface are related by a homotopy (rel endpoints). This is no longer
true when $A$ has annular components, and to account for the
indeterminacies which correspond to changing the homotopy classes
(rel endpoints) of the Whitney disk boundaries on annular components
of $A$ we will include new INT($A$) relations.

Consider any Whitney disk $W_{ij}$ pairing intersections between
an annular component $A_i$ of $A$ and any component $A_j$
(including possibly $j=i$). Since $A_i$ is inessential, push-offs
of either component of $\partial A_i$ bound immersed disks in $X$.
Let $D_{ij}$ be an immersed disk, with $\partial D_{ij}$ a
parallel push-off of a component of $\partial A_i$ sitting in a
collar on $\partial A_i$. The normal disk-bundle of $D_{ij}$ in
$X$ has a nowhere vanishing section given by pushing tangentially
along $A_i$, and we require that this section extends to a nowhere
vanishing section over $D_{ij}$; this can always be arranged by
boundary-twisting $D_{ij}$ if necessary. As illustrated in
Figure~\ref{boundaryINTdiskandRelation-fig}, by half-tubing a
collar of $W_{ij}$ into a collar of $D_{ij}$ along any embedded
arc in $A_i$ we can change the homotopy class of $\partial W_{ij}$
(rel endpoints) in $A_i$ by $\pm$ a generator of $\pi_1A_i$
(meaning that $\partial W_{ij}$ is band-summed into a loop which
pulls back to a generator of the fundamental group of the domain
of $A_i$).
\begin{figure}
\centerline{\includegraphics[width=125mm]{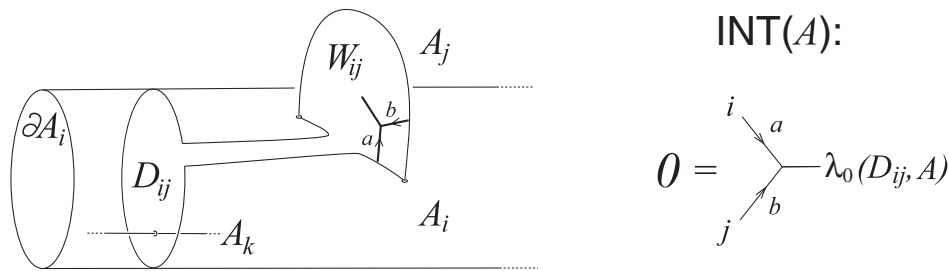}}
         \caption{Changing the homotopy class in $A_i$ of the boundary of a Whitney disk
         leads to new INT($A$) relations determined by the order zero
         intersection invariant $\lambda_0(D_{ij},A)$, with the corresponding univalent labels
         taken from the components of $A$.}
         \label{boundaryINTdiskandRelation-fig}
\end{figure}

The framing requirement on $D_{ij}$ means that this operation
preserves the framing on $W_{ij}$, and the resulting change in the
contribution of $W_{ij}$ to $\tau_1(A)$ can be expressed as the
relation illustrated in
Figure~\ref{boundaryINTdiskandRelation-fig}, where $a$ and $b$ are
the descending group elements for $W_{ij}$, and the order zero
invariant $\lambda_0(D_{ij},A)$ counts only intersections between
$D_{ij}$ and all the $A_k$ (not intersections among the $A_k$), as
was the case for the 2-spheres $S_{ij}$ in the original INT($A$)
relations of Figure~\ref{INTrelationTubed2sphere-fig}.

By the proof of the independence of $\tau_1$ on the Whitney disk
interiors (\ref{subsec:htpy-invariance-of-tau-proof}), using the
usual INT($A$) relations, this expression does not depend on the
choice of $D_{ij}$ (or on the choice of the component of $\partial
A_i$). Since the complement of any number of embedded arcs in an
annulus is path-connected, this operation can be iterated any
number of times on any number of Whitney disks, to get any choices
of homotopy classes (rel endpoints) for the Whitney disk
boundaries, so all the resulting indeterminacies are linear
combinations of the expressions in
Figure~\ref{boundaryINTdiskandRelation-fig}. By including these
expressions into the INT($A)$ relations, and otherwise defining
$\tau_1(A)$ as before, we get a well-defined homotopy invariant,
with $A$ allowed to have inessential properly immersed annular
components.

We remark that the original INT($A$) relations are determined by
computing order zero intersection invariants for a generating set
for $\pi_2X$ (as a module over $\pi_1X$). These extended INT($A$)
relations are determined by computing finitely many more order
zero intersection invariants, which in our applications to links
will essentially correspond to the order zero intersection tree of
the link.

These new INT($A$) relations can realized by performing
finger-moves on $A$ to create clean Whitney disks with prescribed
group elements, and then performing the above operation. Since the
usual geometric manipulations of Whitney towers -- such as
``transferring'' intersections between Whitney disks (Figure~10 in
\cite{ST1}) -- also work for annular components, it follows that
the geometric statements in Theorem~\ref{thm:tau} hold for this
generalized version of $\tau_1(A)$.

Summarizing, we have:
\begin{thm}\label{thm:tau-for-inessential-A}
Theorem~\ref{thm:tau} and Corollary~\ref{cor:tau=stable embedding}
also hold if any components of $A$ are inessential properly
immersed annuli.
\end{thm}

\subsection{The order 1 intersection pairing of links and
spheres}\label{subsec:link-sphere-pairing} In this subsection we
use Theorem~\ref{thm:tau-for-inessential-A} to define an order 1
intersection pairing between links and 2-spheres which determines
the indeterminacy subgroup $\Phi(z)$ of
Theorem~\ref{thm:tau(Lz,L)}.

Let $L$ be any link of null-homotopic knots, and $S$ be any
immersed 2-sphere in $M$. Perturb $S$ into $M\times I$, and extend
$L$ to the collection of embedded annuli $A=L\times I\subset
M\times I$. By Lemma~\ref{embedded-sphere-lemma}, $S$ is homotopic
to an embedding, and since the components of $L$ are
null-homotopic it follows that $\tau_0(A,S)=0$. So we can define
the pairing
$$
\tau_1(L,S):=\tau_1(A,S)\in\cT_1(\pi_1M)/\mbox{INT($z$)}
$$
where we write INT($z$) for INT($A,S$), since (by
Lemma~\ref{embedded-sphere-lemma} and the fact that the components
of $L$ are null-homotopic) the only INT relations come from the
boundary $L$, and are determined by $z=\tau_0(L)$ as follows.
Using the ordering of the link components we can write $\tau_0(L)$
in a \emph{normal form}:
$$
\tau_0(L)=z=\sum_{i\leq j}z_{ij}=\sum_{i\leq j}\sum_p
e_p\cdot(g_p)_{ij},
$$ where the coefficients $e_p$ are non-zero integers, with the understanding that
the indices $p$ depend on $i,j$ -- this is just collecting
together all the intersections between $D_i$ and $D_j$ that have
the same group element $g_p$. The group elements are assumed to be
distinct, i.e. $(g_p)_{ij}\neq (g_{p'})_{ij}$ for $p\neq p'$; and
the only indeterminacy in this normal form for $z$ is in case
$i=j$ we choose only one of $(g_p)_{ii}$ or $(g_p^{-1})_{ii}$.

Then each INT relation corresponding to a 2-disk $D_{ij}$ as in
Figure~\ref{boundaryINTdiskandRelation-fig}, with the group
element $a$ normalized to $1$ by a HOL relation, can be written:
$$
0=(1,g,\sum_{k\neq
i}z_{ki})_{ijk}+(1,g,z_{ii}+\overline{z_{ii}})_{iji}
$$
where the first term is determined by order zero linking between
$L_i$ and the other components $L_k$ for $k\neq i$; and the second
term is determined by order zero linking between $L_i$ and a
0-parallel push-off $L'_i$. Note that the second term is
independent of the above mentioned indeterminacy in the normal
form of $z$ (it corresponds to the relation
$\lambda(D,D')=\mu(D)+\overline{\mu(D)}$ in Wall's intersection
form for $D'$ a zero-parallel push-off of $D$ -- see the next
example).

\subsubsection{Example}\label{S1crossS2pairingexample}
This subsubsection computes the pairing $\tau_1(L,S)$ in a case
where $L$ is a knot, illustrating its relation to the study of $L$
up to stable concordance by defining the indeterminacy subgroup
$\Phi(z)$ and target group of Theorem~\ref{thm:tau(Lz,L)} in a
simple example. It also serves as motivation and a warm-up for the
more complicated computations in the subsequent two subsections.

Consider the case $M=S^1\times S^2$, with $\pi_1M\cong\Z$
generated multiplicatively by $x$. Let $x^nS$ denote a
cross-sectional 2-sphere $S$ equipped with a whisker corresponding
to $x^n\in\pi_1M$, and let $L$ be the clasp knot $K\subset M$ as
in Figure~\ref{WH-double-AND-framed-in-S1crossS2-fig} with
$\tau_0(K)=x\in\cT_0(\langle x\rangle)$. (In this case $K$ is just
a positive Whitehead double of a circle factor representing the
generator $x$; by the ``belt trick'' there are two such Whitehead
doubles up to isotopy.)
\begin{figure}
\centerline{\includegraphics[width=135mm]{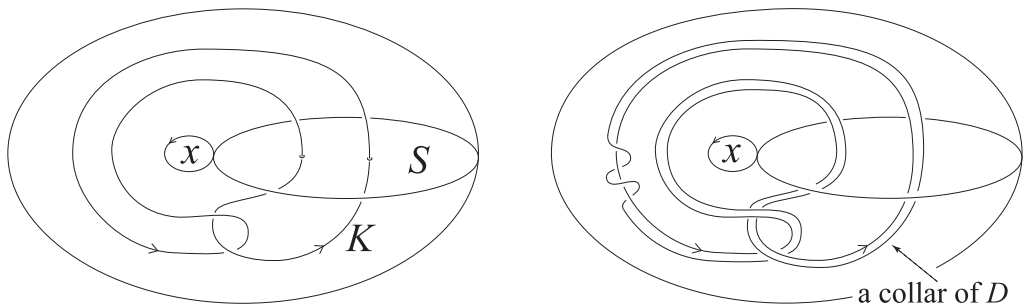}}
         \caption{On the left, the clasp knot $K$ and (one hemisphere of)
         the cross-sectional 2-sphere $S$ in
         $M=S^1\times S^2$ (with only $S^1\times D^2$ shown).
         On the right, a collar of a 0-framed immersed 2-disk $D$ bounded by
         a parallel copy of $K$ in $K\times I$.}
         \label{WH-double-AND-framed-in-S1crossS2-fig}
\end{figure}

The target for $\tau_1(K,x^nS)$ is the quotient of $\cT_1(\langle
x\rangle)$, with univalent labels from $\{K,S\}$, by the
intersection relations $\mbox{INT}(K\times I,S)=\mbox{INT}(x)$ which
can be computed explicitly from intersecting a (framed)
null-homotopy $D$ of a parallel copy of $K$ with $K\times I$
(crossing changes in the right hand side of
Figure~\ref{WH-double-AND-framed-in-S1crossS2-fig}):
$$
\mbox{INT}(x)=(1,x^r,x+x^{-1}-2)_{KjK}=(1,x^r,x+x^{-1})_{KjK}=0
$$
for $j\in\{K,S\}$ and all $r\in\Z$. (Note that the contributions
$(1,x^r,-2)=-2(1,x^r,1)$ from the two negative twists which preserve
the framing of $D$ are killed by the AS relations.)

As illustrated in
Figure~\ref{WH-double-in-S1crossS2withWdisk-fig}, $K\times I$ and
$x^nS$ support an order 1 Whitney tower $\cW$ with a single
Whitney disk, and
$$
\tau_1(\cW)=(1,x^{-1},x^n)_{KKS}=(x,1,x^{n+1})_{KKS}=-(1,x,x^{n+1})_{KKS}
$$
where the middle equality comes from the HOL relations, and the
right-most equality comes from the AS relations.
\begin{figure}
\centerline{\includegraphics[width=100mm]{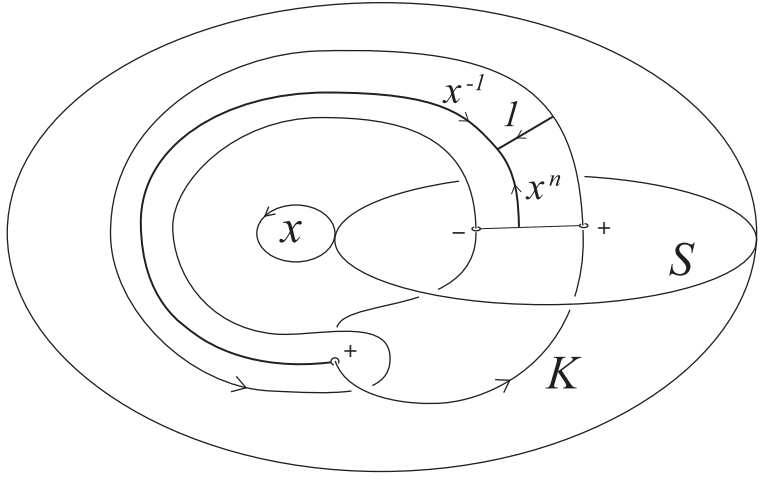}}
         \caption{The order 1 Whitney tower on $K\times I$ and
          $x^nS$. (Whiskers are not shown.)}
         \label{WH-double-in-S1crossS2withWdisk-fig}
\end{figure}

It follows that $\tau_1(K,x^nS)$ lies in the span of Y-trees
having two univalent vertices labeled by $K$ and one by $S$. Using
the HOL relation to trivialize one $K$-labeled edge decoration, we
can write any such Y-tree as
$$
(x^l,x^m):=(1,x^l,x^m)_{KKS}\quad\quad l,m\in\Z.
$$
By the AS and HOL relations we have
$$
(x^l,x^m)=-(x^{-l},x^{m-l}),
$$
so the relevant subgroup of $\cT_1(\langle x\rangle)$ is
isomorphic to $\Z^\infty\oplus \Z_2^\infty$, with the $\Z$ factors
generated by $(x^l,x^m)$, for $l\in\N$ and $m\in\Z$, and the $Z_2$
factors generated by $(1,x^m)$, for $m\in\Z$.

Now, in terms of these generators the relevant INT($x$) relations
$(x+x^{-1},x^r)=(x,x^r)+(x^{-1},x^r)=0$ give:
$$
(x,x^r)=(x,x^{r+1})
$$ for all $r\in\Z$ (using $(x^{-1},x^r)=-(x,x^{r+1})$ by HOL and AS).

So for all $n$ we have
$$
\tau_1(K,x^nS)=-(x,x^{n+1})=-(x,x)\neq 0\in\cT_1(\langle
x\rangle)/\mbox{INT}(x).
$$

The discussion so far has focused on computing $\tau_1(K,x^nS)$.
Recall that the relevance to the knot theory of $K$ from our point
of view is that $x^nS$ can be tubed into any singular concordance
$A$ of $K$, and the corresponding indeterminacy in $\tau_1(A)$ lives
in $\cT_1(\langle x\rangle)/\mbox{INT}(x)$ with \emph{all univalent
labels corresponding to $K$}. These indeterminacies are described
algebraically by just replacing the $S$-labels with $K$-labels in
$\tau_1(K,x^nS)$, which gives the indeterminacy subgroup
$\Phi(x)<\cT_1(\langle x\rangle)/\mbox{INT}(x)$, as $x^n$ varies
over $\pi_1M$ with the generator $S$ of $\pi_2M$ fixed.

Specifically, we can describe the quotient of $\cT_1(\langle
x\rangle)/\mbox{INT}(x)$ by $\Phi(x)$ as follows:

Using the natural ordering of $\langle x \rangle\cong\Z$ with the
chosen generator $x$, we can write the generators of
$\cT_1(\langle x\rangle)$ with all univalent vertices labelled by
$K$ in normal form (as in \ref{subsubsec:T1notation} above):
$$
\cT_1(\langle x\rangle)=\langle(x^m,x^n)\rangle_{1\leq
m<n}\oplus\langle(x^r,x^r)\rangle_{0\leq r}\cong
\Z^\infty\oplus\Z_2^\infty.
$$

Expressing the projected INT($x$) relations
$(x,x^r)+(x^{-1},x^r)=0$ for $r\in\Z$ in the normal form
generators gives the relations:
\begin{equation}\label{int-rel-terms}
(x,x^r)=(x,x^{r+1})\quad 1 \leq r\quad\mbox{and}\quad
(x^q,x^{q+1})=(x^{q+1},x^{q+2})\quad 0\leq q
\end{equation}
which intersect when $r=2$ and $q=1$, so the INT($x$) relations
identify an infinite family of normal form generators.

Now the projected image $\Phi(x)$ of each of the pairings
$\tau_1(K,x^nS)=-(x,x^{n+1})=-(x,x)$, for all $n$, is equal to (the
inverse of) one of the generators identified by the INT relations,
so the quotient of $\cT_1(\langle x\rangle)/\mbox{INT}(x)$ by
$\Phi(x)$ is gotten by setting all the terms in equation
\eqref{int-rel-terms} equal to zero.

For example, by Theorem~\ref{thm:tau(Lz,L)} any combination of
Borromean surgeries on $K$ which correspond to linear combinations
of terms in equation~\eqref{int-rel-terms} do not change the stable
concordance class of $K$. On other hand, there still exists
$\Z^\infty\oplus\Z_2^\infty$ many distinct stable concordance
classes of knots which are order 1 Whitney concordant to $K$, as
detected by the generators $(x^m,x^n)$ with $|m-n|\geq 2$ and $m\neq
1$.

\subsection{Computing
$\tau_1(L,S)$}\label{subsec:computing-link-sphere-pairing} The
computations of this subsection are used in the following
subsection to describe the indeterminacy subgroup $\Phi(z)$ of
Theorem~\ref{thm:tau(Lz,L)}.

It follows from Theorem~\ref{thm:tau-for-inessential-A} that
$\tau_1(L,S)$ is invariant under isotopy of $L$ and homotopy of $S$
in $M$. In (the proof of) Lemma~\ref{lem:link-sphere-pairing} below
we will see how $\tau_1(L,S)$ is invariant under order 1 Whitney
concordance of $L$, and is completely determined by $S$ and
$\tau_0(L)$.

In order to express $\tau_1(L,S)$ as a function of $S$ and
$\tau_0(L)$ it will be useful to define a map $\sigma_S:\pi_1M
\rightarrow\Z[\pi_1M]$ for each $S$ representing an element of
$\pi_2M$. This will be done by defining the map on each generator
of $\pi_2M$ as a $\pi_1M$-module, then extending linearly. We may
ignore any generators which are spherical boundary components of
$M$, since they clearly will have vanishing $\tau_1(L,S)$ for any
$L$. The remaining generators consist of 2--spheres that decompose
$M$ into prime factors, and cross-sections of any $S^1 \times S^2$
factors. The maps are defined slightly differently for each of
these two types of spheres. (The first-time reader could at this
point jump to Lemma~\ref{lem:link-sphere-pairing}, and then begin
to absorb the definition of $\sigma_S$ during the subsequent
proof.)

\subsubsection{Separating spheres}\label{subsubsec:separating-spheres-sigma}
Let $S$ be an embedded separating sphere giving a connected sum
decomposition of $M$ as $M=M'\sharp_S M''$. Any $g\in\pi_1M$ can
be written $g=h_0\prod_{l=1}^n g_l h_l$, where the $g_l$ are
carried by $M'$ and the $h_l$ are carried by $M''$. Denote by
$\beta_q:=h_q\prod_{l=q+1}^n g_l h_l$, with $\beta_n=h_n$, and
define $\sigma_S(g)$ by
$$
\sigma_S(g):=\sum_{q=1}^n (g_q-1)\beta_q \in \Z[\pi_1M].
$$
Thus, here $\sigma_S(g)$ is roughly an alternating sum of
``tails'' of $g$. For example, in the case $n=2$ we have:
$$
\sigma_S:g=h_0g_1h_1g_2h_2 \mapsto (g_1-1)(h_1g_2h_2)+(g_2-1)h_2.
$$

\subsubsection{Non-separating spheres}\label{subsubsec:non-separating-spheres-sigma}
The map $\sigma_S$ is
similar but slightly more complicated for non-separating spheres.
Let $S$ be a spherical section of the $S^1 \times S^2$ factor in a
connected sum decomposition of $M$ as $M=(S^1 \times S^2)\sharp
M'$. Any $g\in\pi_1M$ can be written $g=h_0\prod_{l=1}^n
x^{\epsilon_l r_l} h_l$, where $x$ is represented by the circle in
the $S^1 \times S^2$ factor, with $\epsilon_l=\pm 1$ and the $r_l$
positive integers, and the $h_l$ are carried by $M'$. Let
$\beta_q:=h_q\prod_{l=q+1}^n x^{\epsilon_l r_l} h_l$, with
$\beta_n=h_n$, and define $\sigma_S(g)$ by
$$
\sigma_S(g):=\sum_{q=1}^n \epsilon_q\,
(\sum_{l=1}^{r_q}x^{\epsilon_q(l+\frac{\epsilon_q-1}{2})})\beta_q
\in \Z[\pi_1M].
$$
For example, in the case $n=2$, $r_1=3=r_2$,
$\epsilon_1=1=-\epsilon_2$ we have:
$$
\sigma_S:g=h_0x^3h_1x^{-3}h_2 \mapsto
(x^3+x^2+x)h_1x^{-3}h_2-(1+x^{-1}+x^{-2})h_2.
$$
So for separating spheres $\sigma_S(g)$ is again a sum of tails of
$g$, but multiplied by the indicated unit-coefficient polynomials
in $x$ or $x^{-1}$ .

\subsubsection{}

Having defined $\sigma_S$ on generators, the definition for
arbitrary elements of $\pi_2M$ is gotten by extending linearly:
$\sigma_{(aS\pm a'S')}(g):=a\sigma_S(g)\pm a'\sigma_{S'}(g)$.

\begin{lem}\label{lem:link-sphere-pairing}
Let $S$ be any immersed 2-sphere in $M$. For any
$z\in\cT_0(\pi_1M)$, and any link $L \subset M$ of null-homotopic
knots with $\tau_0(L)=z$,
$$
\tau_1(L,S)=\sum_{i\leq j}\sum_p
e_p\cdot(g_p,1,\sigma_S(g_p))_{ijS}\in\cT_1(\pi_1M)/\mbox{INT}(z)
$$
where $z=\sum_{i\leq j}\sum_p e_p\cdot(g_p)_{ij}$.
\end{lem}
Here $z$ is written in normal form as in
subsection~\ref{subsec:link-sphere-pairing} above.
\begin{proof}
First of all, we claim that if $L'$ is any other link with
$\tau_0(L')=z$, then $\tau_1(L',S)=\tau_1(L,S)$:

As mentioned above in subsection~\ref{subsec:link-sphere-pairing},
the INT($L'$) and INT($L$) relations coincide, since they only
depend $z$. The equality of $\tau_0$ implies that there exists a
homotopy from $L'$ to $L$ whose trace is a collection of properly
immersed annuli $A'\looparrowright M\times I$ supporting an order
1 Whitney tower, and connecting the link components in opposite
ends of $M\times I$ (see \ref{subsec:order1wconc} below). By
running this homotopy partially forward, then waiting, and then
backward, in increasing increments, one can construct a homotopy
(rel boundary) from $L'\times I\subset M\times I$ to $A=A'+
(L\times I) -A'\looparrowright M\times I$, where we write the
composition of homotopies additively from left to right, and $-A'$
is the oppositely oriented $A'$ which runs from $L$ back to $L'$
(see Figure~\ref{LSintHtpy-fig}).
\begin{figure}
\centerline{\includegraphics[width=135mm]{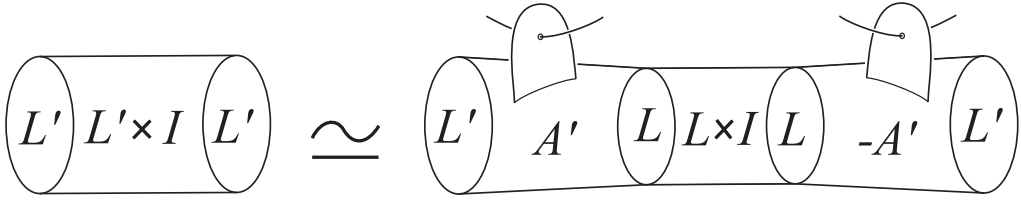}}
         \caption{If $L'$ is homotopic to $L$ by $A'$,
         then $L'\times I\subset M\times I$ is homotopic to
         $A'+(L\times I) -A'\looparrowright M\times I$.}
         \label{LSintHtpy-fig}
\end{figure}
Since $A'$ and $-A'$ each support an order 1 Whitney tower, $A$
supports an order 1 Whitney tower such that no Whitney disks
intersect the part of $M\times I$ containing $L\times I$. (If
$\tau_0(L')\neq \tau_0(L)$ then this construction would still
yield $A$ admitting an order 1 Whitney tower, but there would
necessarily be cancelling pairs of intersections with one point
from each pair in $A'$ and the other in $-A'$, so that the
corresponding Whitney disks would have to intersect the part of
$M\times I$ containing $L\times I$.) Now $S$ is homotopic into the
part of $M\times I$ containing $L\times I$, where Whitney disks
can be found for all intersections between $S$ and $L\times I$
such that the Whitney disks do not intersect the parts of $M\times
I$ containing $A'$ and $-A'$ (since $\tau_0(L\times I,S)=0$). It
follows from the homotopy invariance of $\tau_1$ that
$\tau_1(L',S)=\tau_1(L,S)$, proving the claim.

So to prove the lemma we are free to use any link $L$ with
$\tau_0(L)=z$. This will not be hard using a clasp link
(Figure~\ref{clasplink-guidingarcs-fig} and
subsection~\ref{subsec:clasp-links} below) whose intersections
with $S$ can be paired by Whitney disks that are contained in
neighborhoods of the arcs guiding the clasps.

\begin{figure}
\centerline{\includegraphics[width=105mm]{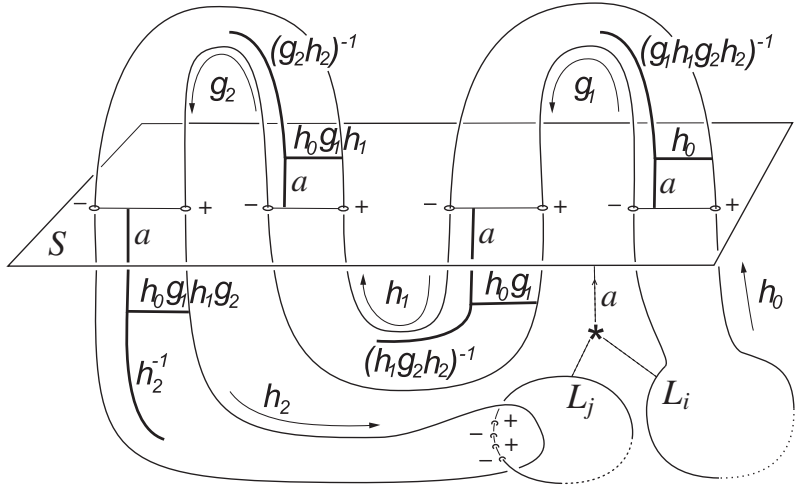}}
         \caption{Four parallel Whitney disks pair intersections between a
         separating sphere $S$ and a clasp link $L$ with $\tau_0(L)=(g)_{ij}$
         for $g=h_0g_1h_1g_2h_2$. The indicated decorations for the associated
         trees can be computed from the diagram, giving
         $\tau_1(L,aS)=(h_0,(g_1h_1g_2h_2)^{-1},a)_{ijS}-(h_0g_1,(h_1g_2h_2)^{-1},a)_{ijS}
         +(h_0g_1h_1,(g_2h_2)^{-1},a)_{ijS}-(h_0g_1h_1g_2,h_2^{-1},a)_{ijS}=
         (g,1,ag_1h_1g_2h_2-ah_1g_2h_2+ag_2h_2-ah_2)_{ijS}$, where the second equality comes from
         the HOL relations.}
         \label{LSintersections3-fig}
\end{figure}
Consider the case where $S$ is a separating sphere giving a
connected sum decomposition $M=M'\sharp_S M''$, and $z=+(g)_{ij}$,
is represented by a single positively signed order zero tree
(edge) oriented from $i$ to $j$, and labeled by $g\in\pi_1M$.
Writing $g=h_0\prod_{l=1}^n g_l h_l$, where the $g_l$ are carried
by $M'$ and the $h_l$ are carried by $M''$, we can create $L$,
with $\tau_0(L)=(g)_{ij}$, from the unlink by introducing a clasp
guided by an arc from the $i$th to the $j$th unlink component,
which intersects $S$ in $2n$ points, so that $A_i=L_i\times I$ has
$2n$ cancelling pairs of intersections with $S$ in some interior
slice of $M\times I$, as illustrated in
Figure~\ref{LSintersections3-fig} for the case $n=2$.

As can be seen in Figure~\ref{LSintersections3-fig}, these
cancelling pairs admit embedded Whitney disks, each of which winds
along the band of the clasp and has a single interior intersection
with $A_j=L_j\times I$. Using the product structure of $M\times I$,
these Whitney disks can be made pairwise disjoint by a small
perturbation. With our usual convention of orienting all edges
towards the trivalent vertex, the group element labels on the trees
associated to these Whitney disks can be seen to have the following
pattern: If the whiskers on the Whitney disks are chosen to respect
the factorization of $g$, then as we move through the Whitney disks
from $A_i$ towards $A_j$ (right to left in the figure), the group
elements labelling the $i$-edge read off increasing sub-words of
$g$, while the group elements labelling the $j$-edge read off
decreasing sub-words of $g^{-1}$, as illustrated in
Figure~\ref{LSintGeneralWdiskPair-fig}.
\begin{figure}
\centerline{\includegraphics[width=105mm]{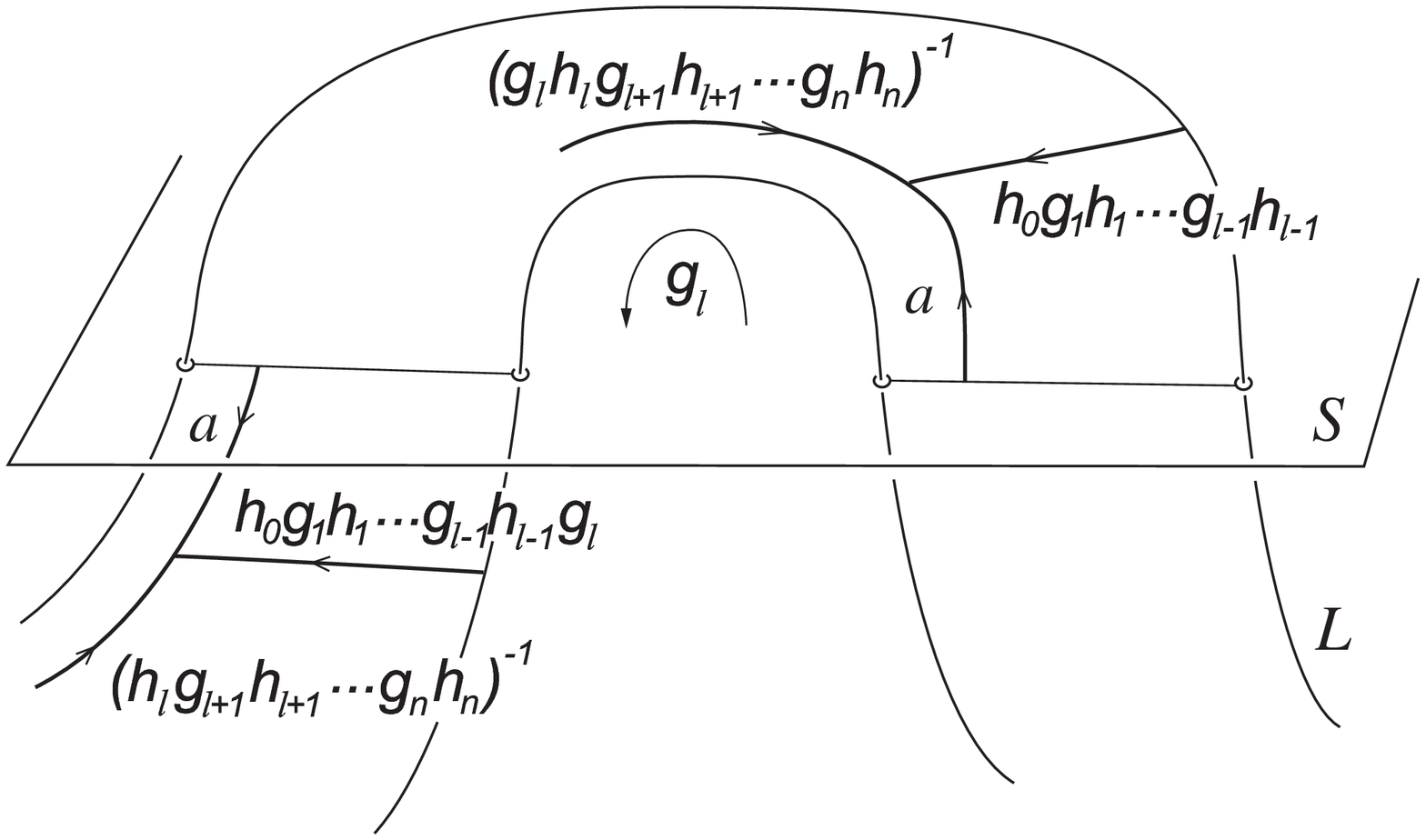}}
         \caption{}
         \label{LSintGeneralWdiskPair-fig}
\end{figure}
A choice of whisker on $S$ determines an element $a\in\pi_1M$
labelling all the edges corresponding to $S$. Now using the right
multiplication of the HOL relations (geometrically, changing the
whiskers on the Whitney disks) to get the group element $g$ on the
$i$-edge of each tree, simultaneously changes each $j$-edge label
to $1$, and changes the $S$-edges exactly as described by the map
$\sigma_S(g)$ (see the example in
\ref{subsubsec:separating-spheres-sigma} above), once the signs
are taken into account, as can be checked in
Figure~\ref{LSintersections3-fig} using our positive corner
convention. Thus in this case, we have
$\tau_1(L,aS)=(g,1,a\sigma_{S}(g))_{ijS}=(g,1,\sigma_{aS}(g))_{ijS}$
as desired.

It is easy to see that $\tau_1(L,S)$ behaves linearly under adding
parallel copies of $S$, with each new copy of $S$ having $2n$ more
parallel Whitney disks, with appropriate sign changes for
orientation-reversed copies of $S$.

Fixing $L$ and $z=(g)_{ij}$, and taking a different
(non-homotopic) separating sphere $S'$, gives a different
factorization of $g$, but the same construction works
simultaneously to confirm the computation for $\pi_1M$-linear
combinations of separating spheres:
$$
\tau_1(L,aS+a'S')=(g,1,a\sigma_S(g))_{ijS}+(g,1,a'\sigma_S'(g))_{ijS'}=(g,1,\sigma_{aS+a'S'}(g)).
$$
And for $z=-(g)_{ij}$, changing the sign of the clasp gives a
global change of sign for $\tau_1(L,S)$.

Since this entire construction takes place in a neighborhood of
the guiding arc for the clasp determining $g$, it can be iterated
any number of times giving the result for $z=\sum_{i\leq j}\sum_p
e _p\cdot(g_p)_{ij}$ and combinations of separating spheres.

\begin{figure}
\centerline{\includegraphics[width=130mm]{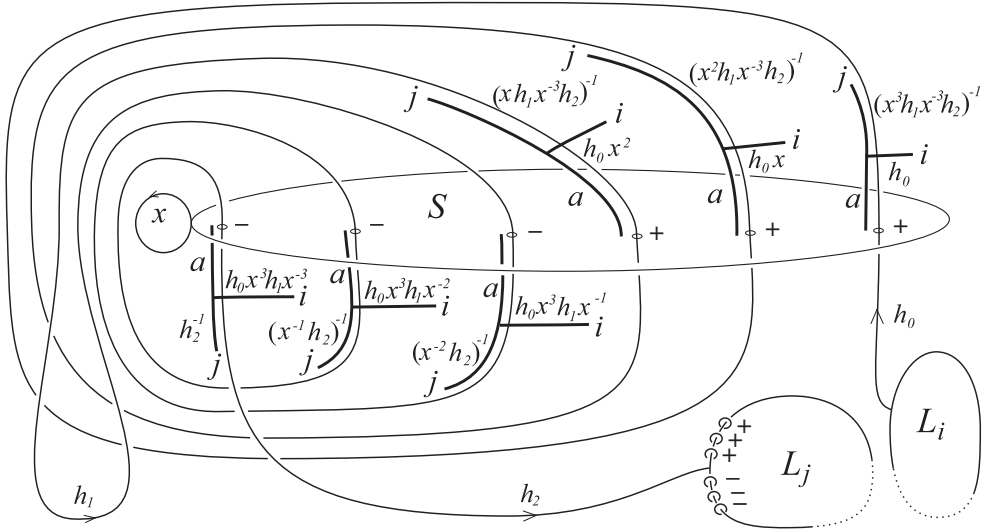}}
         \caption{The arc from $L_i$ to $L_j$ representing $g=h_0x^3h_1x^{-3}h_2$
         guides a clasp which, via the construction of Figure~\ref{LSintersections3-fig},
         leads to six Whitney disks -- one for each intersection between the arc and the
         non-separating sphere $S$.
         Only the trees of the Whitney disks are shown, with edges decorated by group elements
         as computed from the
         diagram.}
         \label{LSintersectionsS2crossS1-fig}
\end{figure}
Checking the case where $S$ is non-separating follows the same
construction: Noting that the construction only depends on the arc
guiding the clasp, Figure~\ref{LSintersectionsS2crossS1-fig} shows
a more schematic illustration of the computation of $\tau_1(z,S)$
for $S$ a cross-section of an $S^1\times S^2$ factor of $M$, in
the case $z=+(g)_{ij}$ with $g=h_0x^3h_1x^{-3}h_2$, where $x$
represents a generator of the circle factor in $\pi_1M$.
Collecting terms and using HOL relations again gives
$\tau_1(L,aS)=(g,1,\sigma_{aS}(g))_{ijS}$ (check with the example
in \ref{subsubsec:non-separating-spheres-sigma} above).
\end{proof}

\subsection{The indeterminacy subgroup
$\Phi(z)$}\label{subsec:Phi-def}

In the definition (\ref{subsec:intro-stable-concordance}) of the
relative order 1 intersection tree $\tau_1(L_z,L):=\tau_1(A)$,
there is an obvious indeterminacy coming from changing the
components of $A$ by connected sums with 2-spheres. By
Lemma~\ref{lem:link-sphere-pairing} this indeterminacy is
characterized by the subgroup $\Phi(z)$ which we define to be the
span of the elements
$$
\sum_{i\leq j}\sum_p
e_p\cdot(g_p,1,a\sigma_S(g_p))_{ijk}\in\cT_1(\pi_1M)/\mbox{INT($z$)}
$$
for $z=\sum_{i\leq j}\sum_p e_p\cdot(g_p)_{ij}$ written in normal
form, where $a$ ranges over $\pi_1M$ and $S$ ranges over a
generating set for $\pi_2M$. Each such element corresponds to the
effect on $\tau_1(A)$ of tubing a 2-sphere $aS$ into the $k$th
component $A_k$.


\section{Proof of
Theorem~\ref{thm:tau(Lz,L)}}\label{sec:tau1(L,Lz)-proof} This
section clarifies the definition of the relative order 1
intersection tree $\tau_1(L_z,L)$ given in
\ref{subsec:intro-stable-concordance}, and introduces notions
relevant to the proof of Theorem~\ref{thm:tau(Lz,L)}. The proof is
based on three lemmas which are left to the end of the section.

We remark that for links in 3-manifolds the equivalence relations
of singular concordance (co-bounding immersed annuli in $M\times
I$) and homotopy (level-preserving singular concordance) coincide
by \cite{Gi,Go}.

\subsection{The $\Psi$ whisker-change
actions}\label{subsec:Psi-stabilizers} The invariants $\tau_0(A)$
and $\tau_1(A)$ depend on fixing whiskers on the $m$ components of
$A$. In the present setting, the corresponding indeterminacies are
characterized by the following (left) actions of the $m$-fold
cartesian product $\Psi:=(\pi_1M)^m$ on $\cT_0(\pi_1M)$ and
$\cT_1(\pi_1M)$:

The action of $\psi=(\psi_1,\psi_2,\ldots,\psi_m)\in\Psi$ on
$\cT_0(\pi_1M)$ and $\cT_1(\pi_1M)$, respectively, is defined on
generators by
$$
(g_p)_{ij}\mapsto(\psi_i g_p\psi_j^{-1})_{ij}
$$
and
$$
(a,b,c)_{ijk}\mapsto(\psi_i a, \psi_j b, \psi_k b)_{ijk}.
$$

\subsection{Latitudes of singular concordances}\label{subsec:lats}

Let $K\subset M$ be a knot, and let $A\looparrowright M\times I$
be a singular concordance from $K$ to $K'$. A path that goes along
a whisker on $K$ in $M\times \{0\}$, then along $A$ to $K'$, and
then along a whisker on $K'$ in $M\times \{1\}$ is called a
\emph{latitude} of $A$. Via projection to $M\times\{0\}$, a
latitude of $A$ determines an element in $\pi_1M$.

For a singular concordance $A$ of a link $L\subset M$, latitudes
of the components determine an element $\psi\in\Psi$.

We write compositions of singular concordances additively from
left to right. Note that if $A_\psi$ and $A_\phi$ have latitudes
determining $\psi$ and $\phi$ respectively, then $A_\psi+A_\phi$
has latitudes determining the product $\psi\phi$. Also, if
$-A_\psi$ denotes the inversion of $A_\psi$, then $-A_\psi$ has
latitudes determining $\psi^{-1}$. Basepoints are assumed to be
taken in $M\times\{0\}$ (unless otherwise stated), so for $n=0,1$
we have:
$$
\tau_n(-A_\psi)=-\psi^{-1}\cdot\tau_n(A_\psi),
$$
as well as:
$$
\tau_n(A_\psi+ A_\phi)=\tau_n(A_\psi)+\psi \cdot\tau_n(A_\phi)
$$
and
$$
\tau_n(A_\psi- A_\phi)=\tau_n(A_\psi)-\psi\phi^{-1}
\cdot\tau_n(A_\phi)
$$
where the actions of $\Psi$ on $\cT_0$ and $\cT_1$ are both
denoted by ``\,$\cdot$\,''. Note that when inverting a singular
concordance, the orientation of $M\times I$ is assumed to remain
fixed. Specific orientation conventions are suppressed from
notation, but consistent throughout the examples.

\subsection{Order 1 Whitney concordance}\label{subsec:order1wconc}
Consider $L$ and $L'$ with
$\tau_0(L)=z=\tau_0(L')\in\cT_0(\pi_1M)$, for some choices of
whiskers. Placing $L$ and $L'$ in either end of $M\times I$ and
connecting the components of singular null-concordances $D$ and
$D'$ by disjointly embedded tubes gives a singular concordance $A$
from $L$ to $L'$. Using the whiskers on $L\subset M\times\{0\}$,
we have
$$
\tau_0(A)=\tau_0(D)-\psi\cdot\tau_0(D')=z-\psi\cdot z
$$
for some $\psi\in\Psi$ which depends on the whiskers and the
choices of tubes. Since we can re-choose the tubes in the
construction of $A$ to realize any element in $\Psi$, we can
arrange for $\psi$ to be trivial, so that
$\tau_0(A)=0\in\cT_0(\pi_1M)$, which implies that $L$ and $L'$ are
order 1 Whitney concordant.

(Note that $A$ can be constructed to be the trace of a homotopy
from $L$ to $L'$: Take $D$ and $D'$ to be homotopies to the
unlink, and realize the tubes by isotopies of the unlink
components.)

It follows that $\tau_0(L)\in\cT_0(\pi_1M)$ classifies order 1
Whitney concordance for links of null-homotopic knots, since the
converse is given by the last line of the proof of
Theorem~\ref{thm:tau(L)} in \ref{sub-sec:proof-of-thm-tau(L)}.

Note that the same argument shows that $\tau_1(L)\in\cT_1(\pi_1M)$
classifies order 2 Whitney concordance, hence stable concordance,
for $L$ with $\tau_0(L)=0$. But the classification for
$\tau_0(L)\neq 0$ -- as per Theorem~\ref{thm:tau(Lz,L)} -- is more
subtle (as is the classification in the case of essential knots
given in Theorem~\ref{thm:essential-tau(Lz,L)} of
Section~\ref{sec:essential-knots}).

\subsection{Clasp links}\label{subsec:clasp-links} For each
$z\in\cT_0(\pi_1M)$ the link $L_z$ in Theorem~\ref{thm:tau(Lz,L)}
can be taken to be a \emph{clasp link} constructed from an
$m$-component unlink $U$ in the following way. Fix whiskers and
orientations on $U$. For each $e_p\cdot (g_p)_{ij}$ in $z$, choose
$|e_p|$ embedded \emph{guiding arcs} joining $U_i$ and $U_j$ which
represent $g_p\in\pi_1M$. Replace each guiding arc by adding a
band to the $i$th component which ``clasps'' the $j$th component
as illustrated in Figure~\ref{clasplink-guidingarcs-fig}, with the
sign of the clasp taken so that the intersection point in the
trace of the homotopy which undoes the clasp has sign equal to the
sign of $e_p\in\Z$.

For given $z$, any two clasp links $L_z$ and $L'_z$ are order $1$
Whitney equivalent, since $\tau_0(L_z)=z=\tau_0(L'_z)$. Fixing a
chosen clasp link $L_z$ gives the classification of
Theorem~\ref{thm:tau(Lz,L)}. Up to isotopy, there are three kinds
of indeterminacies coming from the choices in the construction of
a clasp link: twisting of the clasps (around the guiding arcs),
knotting and linking of the guiding arcs, and the configurations
of the guiding arc endpoints on $U$. As illustrated in
Section~\ref{sec:examples} below, for any two choices the
difference $\tau_1(L_z,L'_z)$ can be explicitly computed.


\subsection{Stabilizers of the $\Psi$-action on $\cT_0$}
\label{subsec:Psi-stabilizers} The proof of
Theorem~\ref{thm:tau(Lz,L)}, as well as its constraints, will
depend on understanding the stabilizers of the $\Psi$-action on
$\cT_0(\pi_1M)$.

For $z\in\cT_0(\pi_1M)$ written in normal form (as in
subsection~\ref{subsec:link-sphere-pairing}), the action $z\mapsto
\psi\cdot z$ of $\psi=(\psi_1,\psi_2,\ldots,\psi_m)\in\Psi$ is
$$
z=\sum_{i\leq j}\sum_p e_p\cdot(g_p)_{ij}\mapsto\sum_{i\leq
j}\sum_p e_p\cdot(\psi_i g_p\psi_j^{-1})_{ij},
$$
so $\psi$ is in the stabilizer $\Psi_z$ of $z$ if and only if for
all $i$ and $j$
$$
\sum_p e_p\cdot(g_p)_{ij} = \sum_p e_p\cdot(\psi_i
g_p\psi_j^{-1})_{ij},
$$
which means that each pair of factors $\psi_i$ and $\psi_j$
permutes the elements $g_p$, and if $i=j$ possibly sends some
$g_p$ to $g_p^{-1}$. (Of course permuting the $g_p$ in $z_{ij}$ is
not sufficient for $\psi$ to be in $\Psi_z$ since the coefficients
of the permuted elements would also have be equal.)

\begin{def}\label{def:untwisted-stablizer}
A stabilizer $\Psi_z$ is called \emph{untwisted} if for all
$\psi\in\Psi_z$, and all $i$ and $j$,
$$
g_p = \psi_i g_p\psi_j^{-1},
$$
and $\Psi_z$ is called \emph{twisted} otherwise. So untwisted
stabilizers are exactly those which induce the identity map on the
fundamental group elements in $z_{ij}$ for all $i$ and $j$.
\end{def}

Will also refer to an element $\psi$ of $\Psi_z$ as \emph{twisted}
or \emph{untwisted}, depending on whether $\psi$ induces a
non-trivial or trivial permutation of the elements in $z$; so that
$\Psi_z$ is untwisted if every $\psi$ in $\Psi_z$ is untwisted. As
discussed in \ref{subsec:twisted stabilizers} below, in a sense
``most'' $z$ have untwisted stabilizers.


\subsection{$\tau_1(L_z,L)$ is
well-defined.}\label{subsec:tauLzL-well-defined} As discussed in
subsection~\ref{subsec:order1wconc} above, it is natural to
interpret the values of the ``absolute'' invariants $\tau_0(L)$
and $\tau_1(L)$ modulo the respective $\Psi$-actions. This means
that these invariants are completely independent of choices of
whiskers.

In the definition of $\tau_1(L_z,L)$ we consider $L_z$ to be
fixed, including the choice of whiskers. Then, for each $L$ the
value of $\tau_1(L_z,L)$ in the quotient of $\cT_1(\pi_1M)$ by
INT($z$) and $\Phi(z)$ is taken modulo the action of the
stabilizer $\Psi_z$. The reason for this is essentially due to
Lemma~\ref{lem:lat-in-stabilizer} below as will become clear
shortly.

Let $A$ and $A'$ be singular concordances from $L_z$ to $L$
supporting order 1 Whitney towers. Then the composition $A-A'$ is
a singular self-concordance of $L_z$ supporting an order 1 Whitney
tower. By Lemma~\ref{lem:lat-in-stabilizer} below, latitudes of
$A-A'$ determine an element $\psi\in\Psi_z$. Since $\Psi_z$ is
assumed to be untwisted, by Lemma~\ref{lem:untwisted-htpies} below
there exists a self-homotopy $H^0_{\psi^{-1}}$ of $L_z$ with
latitudes determining $\psi^{-1}$ such that
$\tau_1(H^0_{\psi^{-1}})=0\in\cT_1(\pi_1M)$. Now the singular
self-concordance $A-A'+H^0_{\psi^{-1}}$ of $L_z$ has latitudes
determining the trivial element $\psi\psi^{-1}\in\Psi$, so by
Lemma~\ref{lem:straightening-lemma} below, $A-A'+H^0_{\psi^{-1}}$
is regularly homotopic (rel $\partial$) to the connected sum of
$L_z\times I$ with disjointly embedded 2-spheres. By the homotopy
invariance of $\tau_1$ and the definition of $\Phi(z)$
(\ref{subsec:Phi-def}) we have
$$
\tau_1(A-A'+H^0_{\psi^{-1}})=\tau_1(A)-\psi\cdot \tau_1(A')\in
\Phi(z),
$$
so $\tau_1(L_z,L)$ does not depend on the choice of $A$.

\subsection{The equivalence of statements (i)-(iv) in
Theorem~\ref{thm:tau(Lz,L)}}\label{subsec:LzLstatements-proof} We
will show the equivalence of (i) and (iii); the equivalence of
(ii), (iii), and (iv) follows from
Theorem~\ref{thm:tau-for-inessential-A} in
Section~\ref{sec:tau-for-inessential-annuli}.

If $L$ and $L'$ are order 2 Whitney concordant, then they cobound
an immersed annulus $A'$ which supports an order 2 Whitney tower.
So any singular concordance $A$ used to compute $\tau_1(L_z,L)$
can be extended by $A'$ to compute
$\tau_1(L_z,L')=\tau_1(A+A')=\tau_1(L_z,L)$, since $\tau_1(A')$
vanishes.

On the other hand, the equality $\tau_1(L_z,L')=\tau_1(L_z,L)$
means that if $A$ and $A'$ are order 1 Whitney concordances from
$L_z$ to $L$ and $L'$ respectively, then
$\tau_1(A')=\psi\cdot\tau_1(A)$ modulo $\Phi(z)$, for some
$\psi\in\Psi_z$. After realizing INT($z$) relations and taking
connected sums with some 2-spheres (if necessary) we may arrange
that $\tau_1(A')=\psi\cdot\tau_1(A)\in\cT_1(\pi_1M)$. By
Lemma~\ref{lem:untwisted-htpies}, there exists a self-homotopy
$H^0_{\psi^{-1}}$ of $L_z$ with $\tau_1(H^0_{\psi^{-1}})=0$ and
latitudes determining $\psi^{-1}$. Now the composition
$-A+H^0_{\psi^{-1}}+A'$ is an order 1 Whitney concordance from $L$
to $L'$, and taking the whiskers on $L_z$ at the end of $-A$ we
have
$$
\tau_1(-A+H^0_{\psi^{-1}}+A')=-\tau_1(A)+\psi^{-1}\cdot\tau_1(A')=
-\tau_1(A)+\psi^{-1}\psi\cdot\tau_1(A)=0
$$
so by Theorem~\ref{thm:tau-for-inessential-A} it follows that
$-A+H^0_{\psi^{-1}}+A'$ is homotopic to a singular concordance
from $L$ to $L'$ which supports an order 2 Whitney tower.

\subsection{Lemmas}\label{subsec:lemmas}
The proof of Theorem~\ref{thm:tau(Lz,L)} is completed by the
following three lemmas.

\begin{lem}[\cite{S}]\label{lem:straightening-lemma}
If a singular self-concordance $A$ of a knot $K\subset M$ has a
latitude that represents the trivial element $1\in\pi_1M$, then
$A$ is homotopic (rel $\partial$) to the connected sum of an
embedded 2--sphere and the embedded annulus $K \times I$ in $M
\times I$.

Furthermore, if $A$ also supports a Whitney tower of order 1 (or
greater), and $K$ has trivial order zero intersections with
2--spheres in $M$, then the above homotopy (rel $\partial$) can be
taken to be a \emph{regular} homotopy.
\end{lem}
Note that $K$ is not assumed to be null-homotopic.

\begin{proof}
The properly immersed $A$ projects to a map of a torus
$T\rightarrow M=M\times \{0\}$, with the image of a latitude in
$A$ projecting to a loop $C$ in the image of $T$ which is
contractible in $M$. The inverse image $\overline{C}$ of $C$ is
algebraically dual to the inverse image $\overline{K}$ of $K$ in
the domain torus $T$, so after a homotopy we may assume that
$\overline{C}$ is geometrically dual to $\overline{K}$ and
embedded. Then $C$ bounds an immersed 2-disk $D$ in $M$ which
surgers the map $T\rightarrow M$ to a map of a 2-sphere
$S\rightarrow M$. This surgery lifts to a homotopy (rel
$\partial$) of $A$ in $M\times I$ to $K\times I\sharp S'$, where
$S'$ is a lift of $S$. By Lemma~\ref{embedded-sphere-lemma}, $S'$
can be assumed to be embedded.

The second statement of Lemma~\ref{lem:straightening-lemma} follows
from the fact that, up to isotopy, a homotopy of surfaces in a
4-manifold consists of a sequence of finger moves, Whitney moves,
and cusp homotopies. If the cusp homotopies occur in cancelling
pairs, then they can be replaced (up to isotopy) by finger moves
(pairs of births) and Whitney moves (pairs of deaths), so that the
homotopy is a regular homotopy. So, if $K\times I\sharp S'$ and $A$
both have all self-intersections occurring in cancelling pairs -- as
will be the case under the further assumptions -- then the above
homotopy may be arranged to be a regular homotopy.
\end{proof}

\begin{lem}\label{lem:lat-in-stabilizer}
For any $z\in\cT_0(\pi_1M)$, let $L_z\subset M$ be any link such
that $\tau_0(L_z)=z\in\cT_0(\pi_1M)$, and let
$A_{\psi}\looparrowright M\times I$ be any singular
self-concordance of $L_z$, with latitudes determining
$\psi\in\Psi$. Then:

If $A_{\psi}$ admits an order one Whitney tower, then
$\psi\in\Psi_z$.
\end{lem}
\begin{proof}
As described in subsection~\ref{subsec:order1wconc}, there exists
a self-homotopy $H_{\psi}$ of $L_z$ with latitudes determining
$\psi$ such that $\tau_0(H_\psi)=z-\psi\cdot z \in \cT_0(\pi_1M)$.
Since the composition $A_{\psi}-H_{\psi}$ has latitudes
representing the trivial element $\psi\psi^{-1}\in\Psi$, it
follows from Lemma~\ref{lem:straightening-lemma} that
$A_{\psi}-H_{\psi}$ is homotopic (rel $\partial$) to the connected
sum of $L_z\times I$ with some 2--spheres. By the homotopy
invariance of $\tau_0$, and since null-homotopic knots have
trivial order zero intersections with 2--spheres, we have that
$\tau_0(A_{\psi}-H_{\psi})=0\in\cT_0(\pi_1M)$. So, computing in
$\cT_0(\pi_1M)$ we have
$$
0=\tau_0(A_{\psi}-H_{\psi})=\tau_0(A_{\psi})-\tau_0(H_{\psi})=0-(z-\psi\cdot
z)
$$
since $\tau_0(A_\psi)=0$, which shows that $\psi\in\Psi_z$.
\end{proof}

\begin{lem}\label{lem:untwisted-htpies}
Let $L_z\subset M$ be a clasp link with $\tau_0(L_z)=z
\in\cT_0(\pi_1M)$, such that $\Psi_z$ is untwisted. Then for any
$\psi\in\Psi_z$, there exists a self-homotopy $H^0_{\psi}$ of
$L_z$ with latitudes determining $\psi$, and an order one Whitney
tower $\cW$ on $H^0_{\psi}$, such that
$\tau_1(\cW)=0\in\cT_1(\pi_1M)$.
\end{lem}
\begin{proof}
The construction of $H^0_{\psi}$ has three steps, and is based on
the observation that undoing and re-doing clasps describes a
self-homotopy supporting an order 1 Whitney tower whose Whitney
disks are described by the guiding arcs of the clasps
(Figure~\ref{UndoClaspWdisk-fig}):
\begin{figure}
\centerline{\includegraphics[width=135mm]{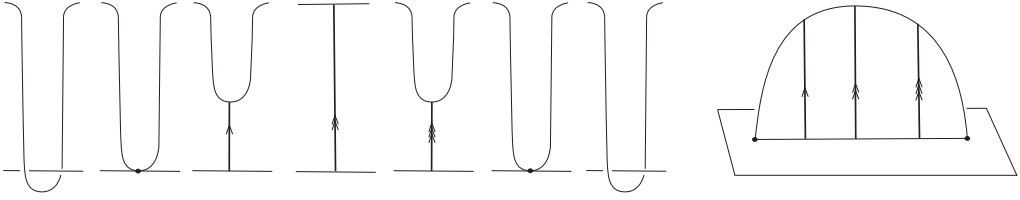}}
         \caption{The guiding arc traces out a Whitney disk as a clasp is
         undone then redone.}
         \label{UndoClaspWdisk-fig}
\end{figure}

First, undoing the clasps by crossing changes describes a homotopy
from $L_z$ to the unlink $U$, with the guiding arcs tracing out
corners of embedded Whitney disks near the crossing-change
intersection points (Figure~\ref{clasplink-guidingarcs-fig}). The
components of $U$ can be assumed to lie in small 3-balls, and the
guiding arcs between $U_i$ and $U_j$ represent the elements
$(g_p)_{ij}$ in $z_{ij}$.

Now, the untwistedness of $\psi$ means that $g_p=\psi_i g_p
\psi^{-1}_j\in\pi_1M$ for all $i$ and $j$ (and $p$ which depends
on $i$ and $j$) which is exactly the condition that a self-isotopy
of $U$ around loops representing $\psi$ extends to self-homotopies
of the guiding arcs and whiskers which are attached to $U$. Each
of these self-homotopies of guiding arcs traces out the center
rectangle of an immersed Whitney disk in $\cW$. The self-isotopy
of $U$ extends to a self-isotopy of the small 3-ball neighborhoods
of the components of $U$ together with a neighborhood of the
whiskers, which forms a thickened wedge of arcs. So the only
singularities created in this second step of the construction are
possible crossing changes between the guiding arcs which are
interior intersections among the Whitney disks of $\cW$.

\begin{figure}
\centerline{\includegraphics[width=140mm]{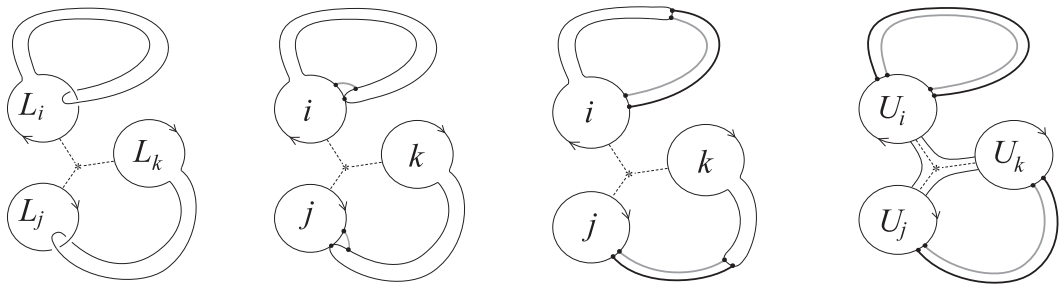}}
         \caption{From left to right, the first part of the homotopy
         $H^0_{\psi}$, and from right to left the third part.
         The guiding arcs trace out Whitney disks, and gray arcs trace out an extension of the
         Whitney section. The right most picture includes the surface $F$.}
         \label{clasplink-htpy-framed-Big-Grey-fig}
\end{figure}

The third step is to run the first homotopy backwards from $U$ to
$L_z$, recreating the clasps by crossing changes, each of which
gives an intersection point which forms a cancelling pair with the
corresponding intersection created when undoing the clasps in the
first step. In this third step the guiding arcs shrink down and
disappear at the intersection points, tracing out the remaining
corners of the Whitney disks in $\cW$.

The union $H^0_{\psi}$ of these three steps is a self-homotopy of
$L_z$ which by construction has latitudes determining $\psi$.
Assuming for the moment that the Whitney disks are framed, it is
clear that $\tau_1(\cW)=0\in\cT_1(\pi_1M)$ because the only
singularities are interior intersections among the Whitney disks
which do not contribute to $\tau_1$ (they are \emph{order 2}
intersections). It remains to show that the Whitney disks can be
arranged to be framed, which will follow essentially from the fact
that Chernov's affine self-linking numbers for framed knots in
3-manifolds vanish modulo 2 (\cite{Ch}).

Recall that the normal section which defines the Whitney disk
framing obstruction is defined by pushing one Whitney disk
boundary arc tangentially along its surface sheet and pushing the
other boundary arc in the normal direction to its surface sheet.
Observe that after a small isotopy ($\pi/2$ radians in the normal
circle bundle) the part of the section that was pushed normally to
a surface sheet will lie in the surface sheet, except for a short
arc near each intersection point (in the right-hand side of
Figure~\ref{Framing-of-Wdisk-fig} picture the back half of the
dotted loop sliding up to the top of the torus). This isotopy does
not change the framing obstruction, and will make it easier to
compute the framings in the construction of $\cW$.

Push-offs of the Whitney disks in $\cW$ extending the Whitney
sections are given by parallel copies of the guiding arcs in the
above description of $H^0_{\psi}$
(Figure~\ref{clasplink-htpy-framed-Big-Grey-fig}). As shown in
Figure~\ref{clasplink-htpy-framed-Big-Grey-fig}, the parts of the
Whitney disks in the first and third homotopies are disjoint from
their push-offs, and so the only possible contributions to framing
obstructions can occur in the second part of the homotopy, which
we will call $H$.

At the start of $H$ we have guiding arcs and their parallel copies
running between components of $U$. The components of $U$ bound
disjointly embedded 2-disks (whose interiors are disjoint from all
guiding arcs), and the guiding arc pairs determine bands between
these 2-disks. After thickening the whiskers on $U$ to a wedge of
bands joined near the basepoint of $M$, the resulting embedded
surface $F$ (shown in the right-most picture of
Figure~\ref{clasplink-htpy-framed-Big-Grey-fig}) gives a normal
framing on each of the (embedded) loops representing the
$(g_p)_{ij}$.

\begin{figure}
\centerline{\includegraphics[width=135mm]{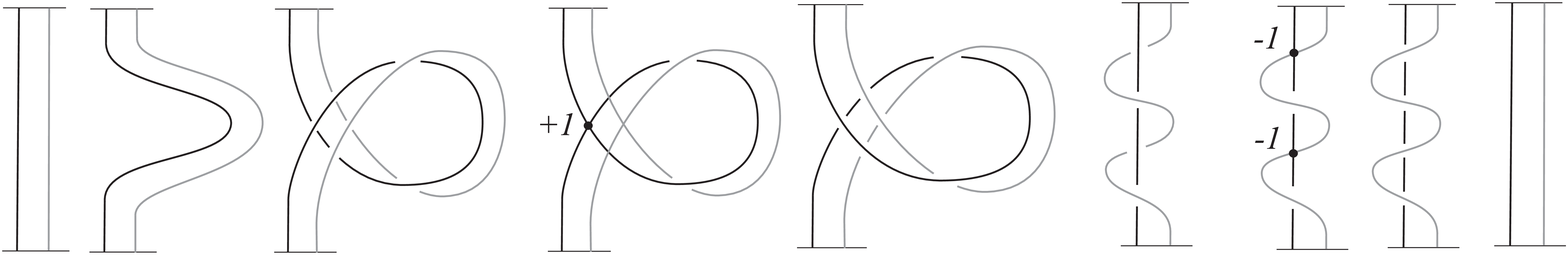}}
         \caption{An \emph{interior twist} on a Whitney disk is the local introduction
         of a self-intersection. Creation of a $\pm 1$-intersection changes the framing by $\mp
         2$, as can be seen in the intersections with a normal push-off of the Whitney disk
         shown in grey. Note that near the self-intersection of the Whitney disk there are a
         pair of intersections between the Whitney disk and its
         push-off which are just artifacts of the immersion of the
         disk bundle and do not contribute to the framing
         obstruction.}
         \label{InteriorTwistPositiveEqualsNegative-fig}
\end{figure}

Now $H$ is a self-homotopy of the union of $U$ together with its
whiskers and the guiding arcs. Clearly $H$ extends to a homotopy
of $F$ (which lies in a neighborhood of $U$ union whiskers and
guiding arcs). We may assume that $H$ restricts to a self-isotopy
on the sub-disk of $F$ which is the union of the disks bounded by
$U$ and the wedge of bands containing the whiskers. At the end of
$H$ the guiding arcs are back where they started, but the parallel
copies might not be, since the twisting of the bands may have
changed (by full twists) during $H$. The framing obstruction on
each Whitney disk in $\cW$ is equal to this change in framing of
the corresponding loop in $F$ at the end of $H$. By the main
results of \cite{Ch}, these relative framings can only differ by
an even integer. So, after perhaps performing interior twists
(which do not contribute to $\tau_1$) as illustrated in
Figure~\ref{InteriorTwistPositiveEqualsNegative-fig}, the Whitney
disks in the above construction of $\cW$ on $H^0_{\psi}$ can
assumed to be framed.
\end{proof}

\section{Examples}\label{sec:examples}
This section computes change-of-base-link formulas for
$\tau_1(L_z,\cdot)$, characterizes twisted stabilizers, and shows
how twisted $\Psi_z$ can lead to non-trivial indeterminacies in
$\tau_1(L_z,\cdot)$.

\subsection{Clasp links}\label{subsec:clasp-link-examples}
For given $z\in\cT_0(\pi_1M)$, a clasp link $L_z$ is determined up
to isotopy by choices at three steps in the construction: the
configuration of endpoints of guiding arcs on the unlink, the
isotopy classes (rel endpoints) of the guiding arcs, and the
(relative) framing of the band corresponding to the thickening of
the guiding arc to a clasp. This subsection illustrates the
computation of the order 1 relative intersection tree
$\tau_1(L_z,L'_z)$ for clasp links $L_z$ and $L'_z$, giving a
``change of base-link'' formula for Theorem~\ref{thm:tau(Lz,L)}.

\subsubsection{Twisting clasps} Figure~\ref{clasp-link-twist-fig}
illustrates the trace $A$ of a homotopy between clasp links $L_z$
and $L'_z$ which differ by one twist of a clasp representing
$g\in\pi_1M$ decorating $(g)_{ji}$ in $z$. The cancelling pair of
intersections of $A$ admit an embedded Whitney disk whose interior
is disjoint from $A$, but which is not framed, as illustrated in
Figure~\ref{clasp-link-unframed-twist-grey-fig}.
Figure~\ref{clasp-link-boundary-twist-grey-fig} shows how an
application of the boundary twist procedure corrects the framing
on the Whitney disk at the cost of creating an interior
intersection between the Whitney disk and $A$, giving
$\tau_1(L_z,L'_z)=(1,1,g)_{iij}$, which is 2-torsion by the AS
relations. In general, twisting such a band $n$ times would give
$\tau_1(L_z,L'_z)=n(1,1,g)_{iij}$.

\begin{figure}
\centerline{\includegraphics[width=115mm]{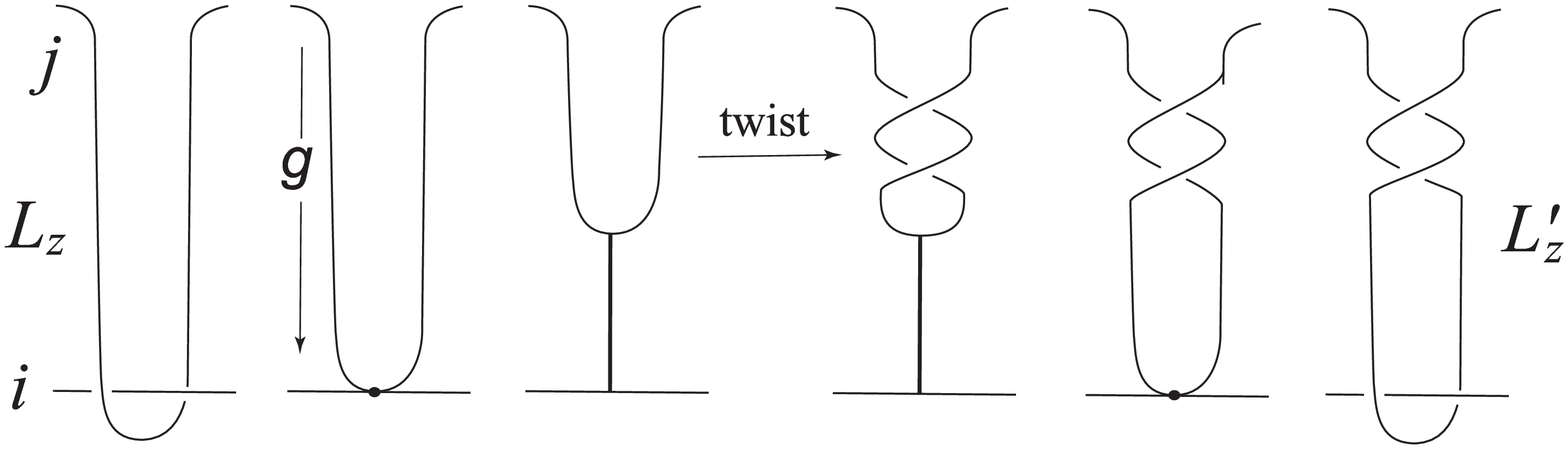}}
         \caption{The trace of a homotopy between clasp links $L_z$ and $L'_z$
         which differ by a single twist in one clasp. The Whitney disk described
          by the vertical arcs is not correctly framed, as shown in Figure~\ref{clasp-link-unframed-twist-grey-fig}.}
         \label{clasp-link-twist-fig}
\end{figure}

\begin{figure}
\centerline{\includegraphics[width=135mm]{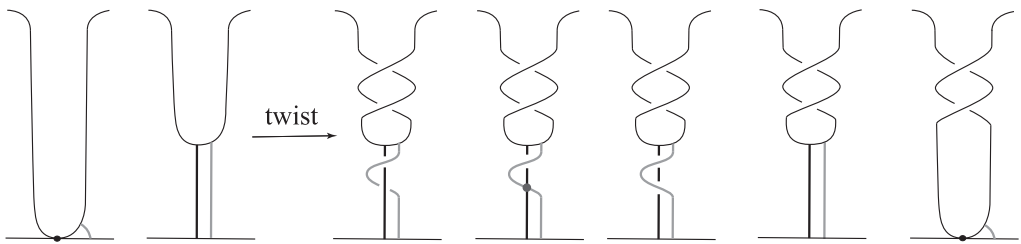}}
         \caption{The grey arcs show an extension of the Whitney section over the interior of the Whitney disk.
         The single intersection between the extension and the Whitney disk indicates that the Whitney disk
         is not framed.}
         \label{clasp-link-unframed-twist-grey-fig}
\end{figure}

\begin{figure}
\centerline{\includegraphics[width=115mm]{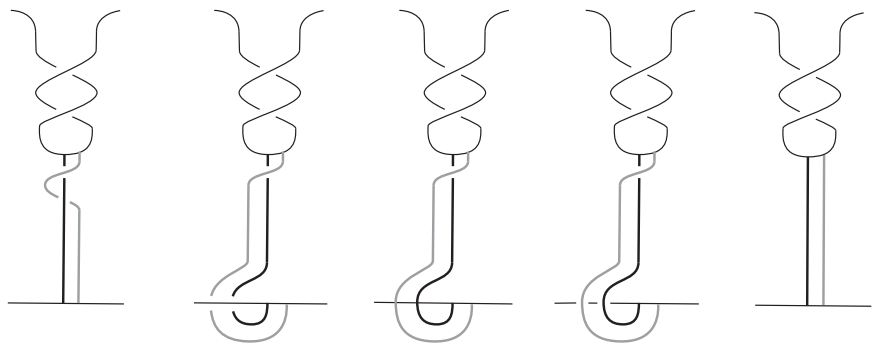}}
         \caption{Changing the Whitney disk of Figures~\ref{clasp-link-twist-fig}
         and \ref{clasp-link-unframed-twist-grey-fig} by the indicated local procedure
         of introducing a boundary twist
         creates a correctly framed Whitney disk at the cost of creating a single
         intersection (center picture)
         between the
         interior of the Whitney disk and the sheet that was twisted around.}
         \label{clasp-link-boundary-twist-grey-fig}
\end{figure}

\begin{figure}
\centerline{\includegraphics[width=115mm]{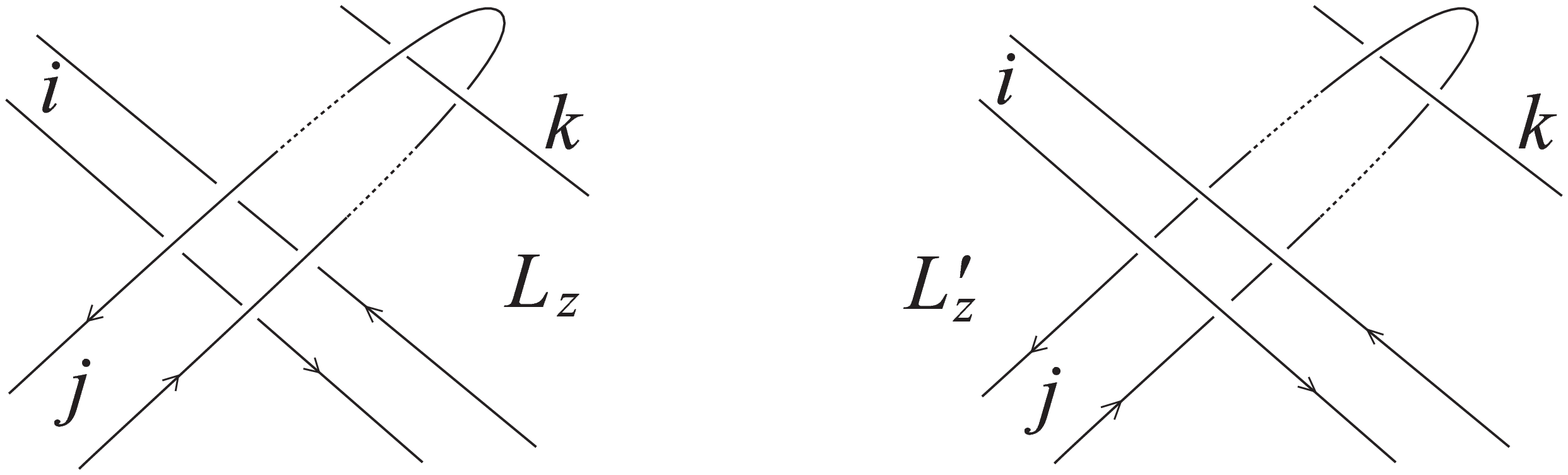}}
         \caption{A homotopy which pushes a clasp of the $i$th component
         across a clasp of the $j$th component gives rise to two parallel
         oppositely oriented Whitney disks (``inside'' the dotted $j$-clasp)
         whose interior intersections with
         the $k$th component will have cancelling trees.}
         \label{clasp-link-band-pass-fig}
\end{figure}

\subsubsection{Choices of guiding arcs} If $L_z$ and $L'_z$ differ
only by a homotopy (rel endpoints) of guiding arcs then
$\tau_1(L_z,L'_z)=0$: This can be seen directly from the ``band
pass'' move illustrated in Figure~\ref{clasp-link-band-pass-fig}.
Or, constructing a homotopy between $L_z$ and $L'_z$ supporting
Whitney disks described by the guiding arcs, the crossings between
guiding arcs correspond to interior intersections between Whitney
disks and do not contribute to $\tau_1$ (as in the proof of
Lemma~\ref{lem:untwisted-htpies}).

\begin{figure}
\centerline{\includegraphics[width=115mm]{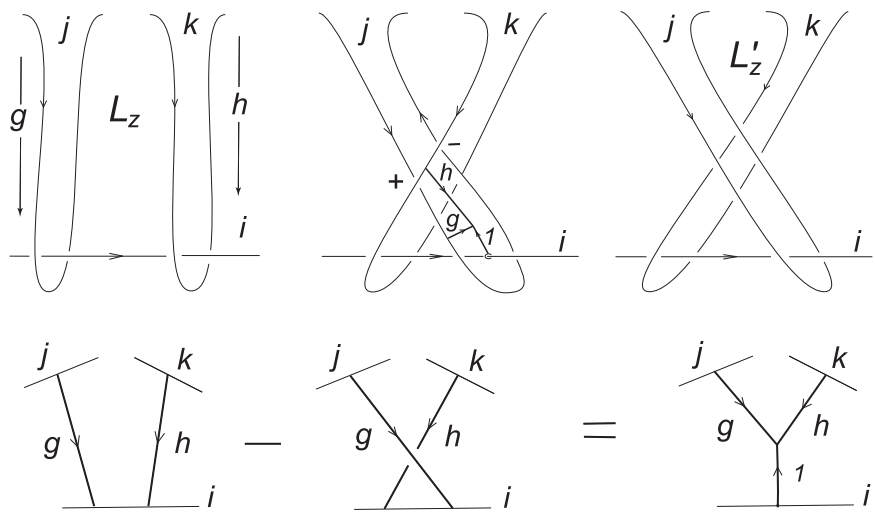}}
         \caption{For clasp links $L_z$ and $L'_z$ differing by the illustrated
         single transposition of guiding arc endpoints, $\tau_1(L_z,L'_z)=(1,g,h)_{ijk}$.}
         \label{STUmove-decorated-fig}
\end{figure}
\subsubsection{Guiding arc endpoint configurations.}
Figure~\ref{STUmove-decorated-fig} shows the computation
$\tau_1(L_z,L'_z)=(1,g,h)_{ijk}$, where $L_z$ and $L'_z$ are
related by a single transposition of guiding arc endpoints along
the $i$th component. The reader familiar with finite type theory
will recognize this as a decorated STU relation.

The transposition of the two endpoints of the \emph{same} guiding
arc is illustrated in Figure~\ref{slide-Wboundary-on-knot-fig},
which computes the corresponding change in $\tau_1$ using the
formula for intersections between Whitney disk boundaries. This
computation will be used in
Example~\ref{subsec:knots-in-B-twisted-S1} below.

\begin{figure}
\centerline{\includegraphics[width=115mm]{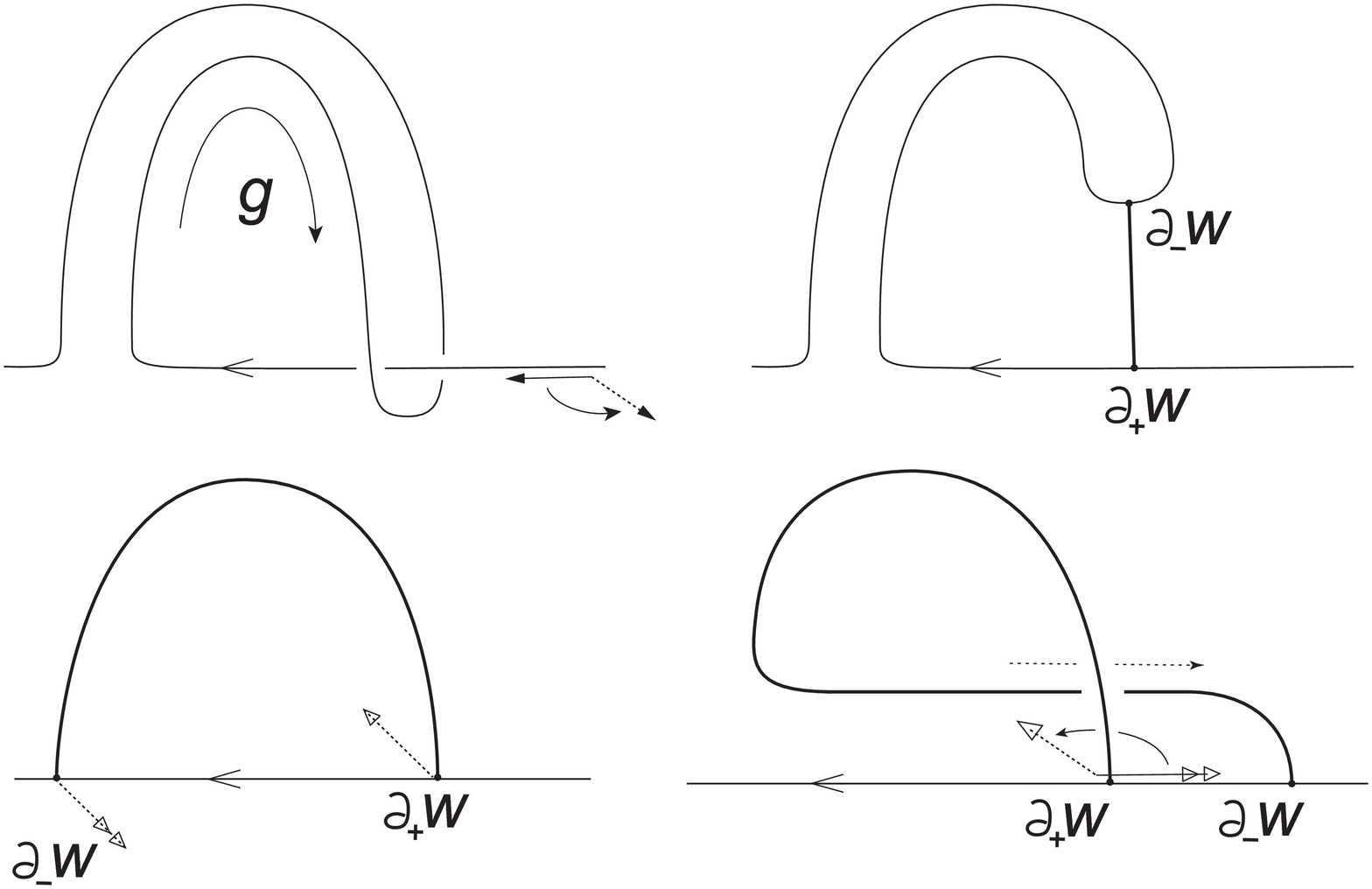}}
         \caption{This figure describes a homotopy $A$ from the clasp knot $K$ shown in
         the upper left to the clasp knot $K'$ which is determined
         by the guiding arc shown in the lower right. The undoing of the positive clasp
         creates a positive self-intersection in $A$, and the redoing of the clasp creates
         a cancelling negative self-intersection. These self-intersections are paired by a framed
         Whitney disk $W$ (described by the trace of the guiding arc) whose
         boundary arcs $\partial_+W$ and
         $\partial_+W$ intersect in a single point $p$ where $\partial_-W$ pushes across $\partial_+W$
         just before the lower right picture as indicated by the horizontal dotted arrow.
         The dotted vectors pointing lower-right point in the direction of increasing time,
         and the dotted vectors pointing upper-left point into the past. The orientation of $A$
         is given by a tangent vector to $K$, together with a second tangent vector pointing into the future.
         The Whitney disk boundary arc $\partial_+W$ is oriented towards the positive
         self-intersection, and the boundary arc $\partial_-W$ is oriented towards the negative
         self-intersection. At $p$ the orientation of $A$ agrees with
         $(\partial_-W,\partial_+W)_p$.
         Using the assignment of trees to intersections between Whitney disk boundaries
         given by the formula $\epsilon_p\cdot t_p:=\epsilon_k \epsilon_j(1,g_k^{\epsilon_k},g_j^{\epsilon_j})$
         in the ``independence of Whitney disk boundaries'' discussion in \ref{subsec:htpy-invariance-of-tau-proof},
         and computing from the figure gives $\tau_1(K,K')=\tau_1(A)=-(g,g^{-1})=+(g^{-1},g)$,
         in the notation $(g,h):=(1,g,h)$.}
         \label{slide-Wboundary-on-knot-fig}
\end{figure}

\subsection{Twisted stabilizers}\label{subsec:twisted stabilizers}
This subsection points out that twisted stabilizers only occur in
the presence of circle bundles over non-orientable surfaces with
orientable total space, or Seifert fibered submanifolds containing
singular fibers, and illustrates how twisted stabilizers can lead
to self-homotopies of knots having non-trivial $\tau_1$.

To present the essential ideas with a minimum of subscripts we
concentrate on the case where $z=\tau_0(K)$ for a knot $K\subset
M$. Identifying $\cT_0(\pi_1M)$ with $\Z[\pi_1M]$ modulo inversion
and trivial elements, and writing $z=\sum e_i g_i$ in normal form
(with all $e_i\neq 0$, $g_i\neq 1$, and with $g_i\neq g_j^{\pm 1}$
for $i\neq j$), any element $\psi$ in the stabilizer $\Psi_z$ of
$z$ satisfies the equation
$$
\sum e_i \psi g_i \psi^{-1} = \sum e_i g_i \quad \quad
\mbox{modulo} \quad g=g^{-1}.
$$
So the conjugation action of $\psi$ permutes the (finite) set
$\{g_i^{\pm1}\}$, and by definition $\psi$ is twisted if this
permutation is non-trivial, and untwisted if this permutation is
the identity.

Assuming that $\psi$ is twisted, there is a (least) natural number
$2\leq n\in\N$ such that $\psi^n$ induces the identity permutation
on $\{g_i^{\pm1}\}$. Writing $\zeta(z):=\cap_i \zeta(g_i)$ for the
intersection of the centralizers $\zeta(g_i)$ in $\pi_1M$ of the
$g_i$, we have $\psi^n\in\zeta(z)$, but $\psi\notin\zeta(z)$. From
work of Jaco-Shalen, Casson-Jungreis, and Gabai, non-cyclic
centralizers of non-trivial elements of $\pi_1M$ are carried by a
codimension zero \emph{characteristic submanifold} of $M$ which is
Seifert fibered (\cite{CJ,G,Ja,JS1,JS2}). Since $n$th roots of any
element are contained in the centralizer of that element, it can
be checked that both $\zeta(z)$ and $\zeta(\psi^n)$ are carried by
the same connected component $N$ of the characteristic submanifold
of $M$.

It follows from section 4 of \cite{JS2} (in particular
Theorem~4.4) that there are only two possible ways for
$\psi^n\in\zeta(z)$, but $\psi\notin\zeta(z)$: Either $\psi$ and
$\{g_i\}$ are all carried by a submanifold $B\widetilde{\times}
S^1 \subset N$ which is the orientable circle bundle over the
M\"{o}bius band, or $\psi$ is carried by a neighborhood of
singular fiber in $N$ and $\{g_i\}$ lie in the canonical subgroup
of $\pi_1N<\pi_1M$ consisting of elements represented by loops
which project to orientation-preserving loops in a base surface
for $N$. In both cases the possible $z$ can be described
explicitly as follows.

\subsubsection{The twisted circle bundle over the M\"{o}bius
band}\label{subsubsec:orientable-S1-over-Mob} In the first case,
we use the notation
$$
\pi_1(B\widetilde{\times}S^1)= \langle \, a,f \, | \,\,
afa^{-1}=f^{-1}\, \rangle<\pi_1M
$$
with $f$ represented by a circle fiber, $a$ represented by a core of
the M\"{o}bius band $B$, and normal form $a^nf^l$ for $n,l\in\Z$. It
is easily checked that the only candidates to be twisted stabilizers
are $\psi=a^{2r+1}f^s$, with $\psi^2\in\langle a^{2k}\rangle$
central. For all $r,s\in\Z$, the orbits under conjugation by
$a^{2r+1}f^s$ are:
$$
\{a^{2k}\}\quad\{f^l,f^{-l}\}\quad\{a^{2k}f^l,a^{2k}f^{-l}\}
$$
for all nonzero integers $k$ and $l$; and
$$
\{a^{2k+1}f^l,a^{2k+1}f^{2s-l}\}\quad\{a^{2k+1},a^{2k+1}f^{2s}\}\quad\{a^{2k+1}f^s\}
$$
for all integers $k$ and $l$ with $l\neq s$. So in this case
$\Psi_z$ is twisted exactly when $z$ is of the form:
\begin{eqnarray*}\label{eqn1}
\lefteqn{ \sum_{l\neq 0, k\neq
0}e_{kl}(a^{2k}f^l+a^{2k}f^{-l})+e_l f^l+e'_k a^{2k}}\\
& & + \sum_{l\neq
s}d_{kl}(a^{2k+1}f^l+a^{2k+1}f^{2s-l})+d_k(a^{2k+1}+a^{2k+1}f^{2s})+d'_k
a^{2k+1}f^s,
\end{eqnarray*}
where the coefficients $e_{kl},e_l,e'_k,d_{kl},d_k,d'_k$ are
integers, with some coefficient other than the $e'_k$ non-zero.

We will show below in \ref{subsec:knots-in-B-twisted-S1} that
there is a self-homotopy $A$ of a knot $K$ having a latitude
determining $a\in\Psi_z$ with $\tau_0(K)=z=f+f^3$ such that
$\tau_1(A)\neq 0\in\cT_1(\pi_1M)/\mbox{INT}(z)$. This illustrates
how Lemma~\ref{lem:untwisted-htpies} (hence the proof of
Theorem~\ref{thm:tau(Lz,L)}) can fail if $\Psi_z$ is twisted.

On the other hand, related examples will show that for $K$ with
$\tau_0(K)=z=e_l f^l$, all self-homotopies $A$ of $K$ have
$\tau_1(A)=0\in\cT_1(\pi_1M)/\mbox{INT}(z)$. Thus, a twisted
$\Psi_z$ does not guarantee the failure of
Lemma~\ref{lem:untwisted-htpies}, and so
Theorem~\ref{thm:tau(Lz,L)} can be extended to allow some twisted
stabilizers.

\subsubsection{Singular Seifert
fibers}\label{subsubsec:singular-fibers-twisting} In the case
where $\psi$ is carried by a neighborhood of singular fiber in $N$
and $\{g_i\}$ lie in the canonical subgroup of $\pi_1N<\pi_1M$, we
have $\psi=\phi^r f^s$ for $r,s\in\Z$ where $f$ is represented by
a regular fiber of $N$ and $\phi$ is represented by a singular
fiber of $N$. In this case $n\in\N$ is the smallest power such
that $\phi^{rn}$ is contained in the cyclic (normal) subgroup
$\langle f \rangle$ of $\pi_1N$. We get fixed $z$ of the form:
$$
\sum e_k(g_k+\phi^r g_k\phi^{-r}+\phi^{2r}
g_k\phi^{-2r}+\cdots+\phi^{(n-1)r}
g_k\phi^{-(n-1)r})+e_{pq}\phi^pf^q
$$
where the $\phi^pf^q$ lie in $\zeta(\phi^r f^s)$, so at least one
$e_k$ must be non-zero.

\subsubsection{Twisted $\Psi_z$ for
links}\label{subsubsec:twisted-Psi-z-for-links} Now let $L$ be a
link with $\tau_0(L)=z=\sum_{i\leq j} z_{ij}$, with each $z_{ij}$
written in normal form:
$$
z_{ij}=\sum_p e_p (g_p)_{ij}.
$$
A twisted $\psi\in\Psi_z$ induces a non-trivial permutation of at
least one of the sets $\{g_p^{\pm 1}\}$ for $i=j$ or $\{g_p\}$ for
$i<j$.

In the cases $i=j$ the above discussion for knots applies to
describe restrictions on $z_{ii}$ and $\psi_i$ for twisted
$\psi\in\Psi_z$. (Of course these restrictions are not independent
of each other in the presence of non-zero $z_{ij}$ with $i<j$.)

Considering the case where $i<j$ and $\psi$ induces a non-trivial
permutation of $\{g_p\}$ (twisting $z_{ij}$), there exists $n\geq
2$ such that $\psi_i^n g_p \psi_j^{-n}=g_p$ for all $g_p$. Thus
$g_p^{-1}\psi_i^n g_p=\psi_j^n$, and $g_p\psi_j^n
g_p^{-1}=\psi_i^n$ for all $g_p$. It follows from Proposition~4.5
of \cite{JS2} that $\psi_i$ and $\psi_j$ are carried by singular
fibers in a characteristic component $N\subset M$.

A more detailed analysis of twisted stabilizers would be
interesting, but we stop here and conclude:
\begin{prop}\label{prop:twisted-stabilizers}
If $M$ contains no orientable circle bundles over non-orientable
surfaces and no Seifert fibered submanifolds having singular
fibers, then $\Psi_z$ is untwisted for all $z\in\cT_0(\pi_1M)$.
\end{prop}

By direct computation one can extend
Lemma~\ref{lem:untwisted-htpies} (hence
Theorem~\ref{thm:tau(Lz,L)}) to certain simple examples of $z$
with twisted $\psi\in\Psi_z$ carried by a singular Seifert fiber,
and it is possible that twisted $\Psi_z$ which correspond only to
singular Seifert fibers do not contribute any non-trivial
indeterminacies to $\tau_1(L_z,\cdot)$, but I do not know of a
general construction. The next example shows that
Lemma~\ref{lem:untwisted-htpies} can indeed fail for twisted
$\Psi_z$ which correspond to orientable circle bundles over
non-orientable surfaces.

\subsection{Knots in the twisted circle bundle over a M\"{o}bius
band.}\label{subsec:knots-in-B-twisted-S1} Using the notation of
subsection~\ref{subsubsec:orientable-S1-over-Mob} above, this
example will consider knots $K$ in $B\widetilde{\times}S^1$ with
$\tau_0(K)=z=\sum e_l f^l$ carried by the circle factor. Such $z$
has twisted $\Psi_z$, with the conjugation action of $a$ (carried
by the M\"{o}bius band core) inverting the $f^l$. After giving an
example of a self-homotopy $A$ of such a $K$ with non-trivial
$\tau_1(A)$ (so that Lemma~\ref{lem:untwisted-htpies} fails), we
will observe that in some cases Lemma~\ref{lem:untwisted-htpies}
can be extended to $K$ with twisted $\Psi_z$.

Throughout this subsection $(g,h):=(1,g,h)$ with the common
univalent labels suppressed.

Since $B\widetilde{\times}S^1$ is irreducible, the $\Phi(z)$
relations are trivial throughout this subsection.

\subsubsection{A self-homotopy with non-trivial
$\tau_1$}\label{subsubsec:non-trivial-self-htpy} Illustrated in
Figure~\ref{twisted-f1-f3-inMobcirclebundle-fig}, using the
unknot-plus-guiding-arcs notation, is a self-homotopy $A$ of a
clasp knot $K$ in $B\widetilde{\times}S^1$, with
$\tau_0(K)=z=f+f^3$. The right-back and left-front faces of the
rectangular prism are identified by a half-rotation, and the top
and bottom faces are identified by a translation, forming
$B\widetilde{\times}S^1$. The clasp knot $K$ is described by the
left-most rectangular unknot $U$ together with two guiding arcs,
one which wraps once around the circle fiber and a longer one
which wraps around three times. The self-homotopy $A$ pushes $K$
along the length of the M\"{o}bius band, so $A$ has a latitude
representing the element $a\in\pi_1 B\widetilde{\times}S^1$.
During the description of $A$, the guiding arcs can be thought of
either as representing the `bands' of $K$, or as tracing out a
pair of Whitney disks in an order 1 Whitney tower on $A$.

\begin{figure}
\centerline{\includegraphics[width=135mm]{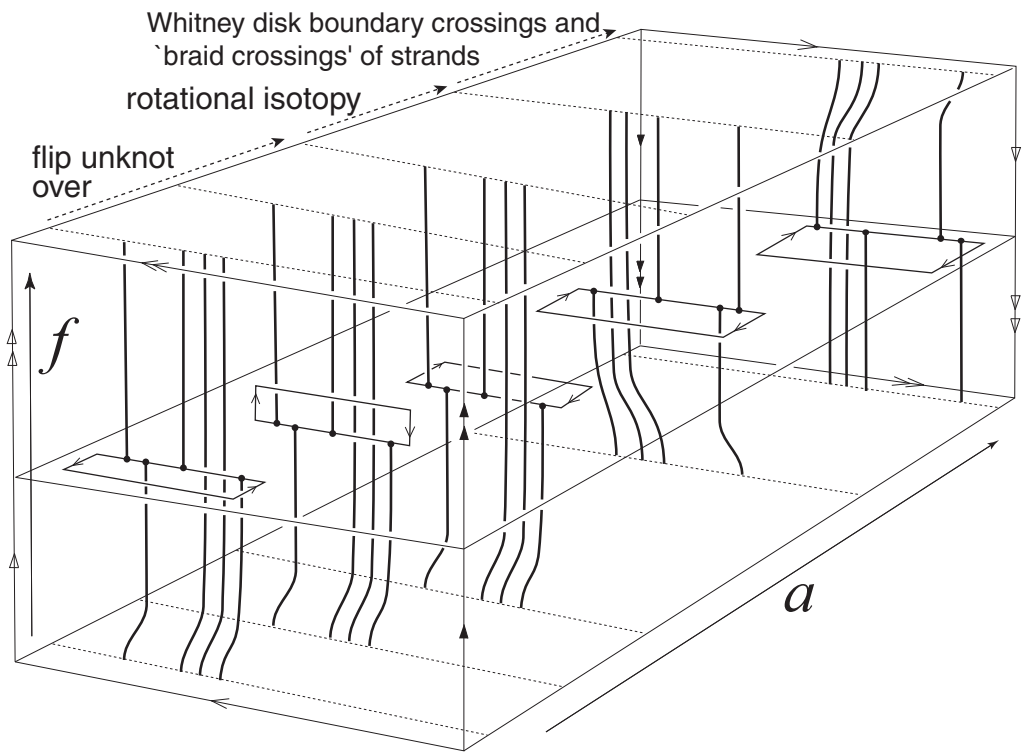}}
         \caption{From left to right are pictured five stages of the self-homotopy $A$ of
         $K$ in the twisted circle bundle $B\widetilde{\times}S^1$ over the M\"{o}bius band.
         During the first three stages, the unknot $U$ flips over while fixing its edge
         that contains the endpoints of the guiding arcs. Between the second and third stages
         $U$ intersects each guiding arc near one endpoint -- contributing
         $(f,f)+(f^3,f^3)$ to $\tau_1(A)$ -- and intersects two strands of the
         longer guiding arc in cancelling pairs of intersections which contribute nothing to $\tau_1(A)$.
         From the third stage to the fourth stage is just an ambient isotopy of $U$ together with the guiding arcs
         which is rotation by $180^\circ$ in the horizontal plane.
         From the fourth stage to the fifth stage each guiding arc
         pushes one of its endpoints across the other -- contributing $(f^{-1},f)+(f^{-3},f^3)$
         to $\tau_1(A)$, as illustrated in Figure~\ref{slide-Wboundary-on-knot-fig} --
         and the vertical strands of the longer guiding arc
         pass through each other contributing nothing to $\tau_1(A)$. From the fifth stage back to the
         first stage is just an isotopy across the back face and in from the front face. Thus
         $\tau_1(A)=(f,f)+(f^{-1},f)+(f^3,f^3)+(f^{-3},f^3)$, which will be seen to be non-zero
         in $\cT_1(\pi_1B\widetilde{\times}S^1)/\mbox{INT}(z)$ (even modulo 2, in case you don't want to check signs).}
         \label{twisted-f1-f3-inMobcirclebundle-fig}
\end{figure}

As explained in the caption of
Figure~\ref{twisted-f1-f3-inMobcirclebundle-fig},
$\tau_1(A)=(f,f)+(f^{-1},f)+(f^3,f^3)+(f^{-3},f^3)$, which we will
show is non-zero in
$\cT_1(\pi_1B\widetilde{\times}S^1)/\mbox{INT}(z)$.

Noting that the edge decorations in $\tau_1(A)$ are carried by
circle fibers, the relevant INT$(z)$ relations for $z=f+f^3$ are:
$$
0=(z+\overline{z},f^r)=(f,f^r)+(f^{-1},f^r)+(f^3,f^r)+(f^{-3},f^r)
\quad\quad \mbox{for} \,\, r\in\Z.
$$
Denoting these relations by $\Delta_r$ and writing them in normal
form $(f^n,f^m)$ with $0 < n < m$ or $0 =n\leq m$, we have
$$
\Delta_r=(f,f^r)-(f,f^{r+1})+(f^3,f^r)-(f^3,f^{r+3})\quad\quad
\mbox{for} \,\, 3<r
$$
$$
\Delta_{-r}=(f^r,f^{r+1})-(f^{r-1},f^r)+(f^r,f^{r+3})-(f^{r-3},f^r)\quad\quad
\mbox{for} \,\, -r\leq -3
$$
$$
\Delta_3=(f,f^3)-(f,f^4)+(1,f^3)-(f^3,f^6)
$$
$$
\Delta_2=(f,f^2)-(f,f^3)-(f^2,f^3)-(f^3,f^5)
$$
$$
\Delta_1=(1,f)-(f,f^2)-(f,f^3)-(f^3,f^4)
$$
$$
\Delta_{-1}=(f,f^2)+(1,f)+(1,f^4)+(f^2,f^3)
$$
$$
\Delta_{-2}=(f^2,f^3)-(f,f^2)+(f^2,f^5)+(f,f^3)
$$
with $\Delta_0$ already zero by the AS and FR relations.

We will show that $\tau_1(A)$ is not equal to a finite sum $\sum
d_i \Delta_i$ in $\cT_1(\pi_1B\widetilde{\times}S^1)$:

In normal form we have
$\tau_1(A)=(1,f)-(f,f^2)+(1,f^3)-(f^3,f^6)$. Assume that
$(1,f)-(f,f^2)+(1,f^3)-(f^3,f^6)=\sum d_i \Delta_i$. Note that for
all $5\leq r$ the coefficient of $(f,f^r)$ must vanish, so
$d_{r-1}+d_r = 0$. Thus, if $d_4$ is odd, then $\sum d_i \Delta_i$
can not be a finite sum. Similarly, for all $5\leq r$ the
coefficient of $(f^{r-1},f^r)$ must vanish, so $d_{-(r-1)}+d_{-r}=
0$. Thus, if $d_{-4}$ is odd, then $\sum d_i \Delta_i$ can not be
a finite sum.

Working modulo 2, the first three terms in $\tau_1(A)$ give the
following relations among the coefficients:
$$
\begin{array}{ccccccc}
  (1,f): & \quad & d_1& + & d_{-1} & \equiv & 1 \\
  (f,f^2): & \quad & d_2 & + & d_{-2} & \equiv & 0 \\
  (1,f^3): & \quad & d_3 & + & d_{-3} & \equiv & 1.
\end{array}
$$
The vanishing of the coefficient of $(f,f^3)$ gives:
$$
d_1+d_2+d_{-2}+d_3\equiv 0
$$
which implies that $d_1\equiv d_3$, since $d_2 + d_{-2}  \equiv
0$. Now the vanishing of the coefficient of $(f^3,f^4)$, together
with the congruence $d_1\equiv d_3$, give:
$$
0 \equiv d_1+d_{-3}+d_4+d_{-4}\equiv d_3+d_{-3}+d_4+d_{-4} \equiv
1 +d_4+d_{-4},
$$
so one of $d_4$ or $d_{-4}$ is odd, and $\sum d_i \Delta_i$ can
not be a finite sum.

\subsubsection{Extending Lemma~\ref{lem:untwisted-htpies} to
some twisted
stabilizers}\label{subsubsec:extending-lemma-to-twisted-case} The
construction of the self-homotopy $A$ in
Figure~\ref{twisted-f1-f3-inMobcirclebundle-fig} generalizes to
clasp knots $K$ with $\tau_0(K)=\sum e_l f^l$, yielding
self-homotopies $A$ with latitudes representing $a$, and
$$
\tau_1(A)=\sum e_l[(f^l,f^l)+(f^{-l},f^l)].
$$
In the case where $\tau_0(K)=z= e_l f^l$, it is clear that
$\tau_1(A)=e_l[(f^l,f^l)+(f^{-l},f^l)]\in\mbox{INT}(z)$ by taking
$r=l$ in the INT relations, and Lemma~\ref{lem:untwisted-htpies}
can be extended to such $K$. (In the case $e_l=1$ so $K$ has just
a single clasp then $A$ can be taken to be a self-isotopy of $K$.)
Notice that composing $A$ with itself in the general case gives a
self-homotopy $A+A$ of $K$ with a latitude representing $a^2$
which is untwisted, and the computation
$$
\begin{array}{rcl}
  \tau_1(A+A) & = & \tau_1(A)+a\cdot\tau_1(A)   \\
   & = & \sum
e_l[(f^l,f^l)+(f^{-l},f^l)+(af^la^{-1},af^la^{-1})+(af^{-l}a^{-1},af^la^{-1})]  \\
  & = & \sum
e_l[(f^l,f^l)+(f^{-l},f^l)+(f^{-l},f^{-l})+(f^l,f^{-l})] \\
& = & 0\in\cT_1(\pi_1B\widetilde{\times}S^1)
\end{array}
$$
confirms Lemma~\ref{lem:untwisted-htpies} (using
$(f^{-l},f^{-l})+(f^l,f^l)=0$ and $(f^l,f^{-l})+(f^{-l},f^l)=0$,
by HOL, FR and AS relations).

General techniques for computing in the infinitely generated
groups $\cT_1(\pi_1 B\widetilde{\times}S^1)/\mbox{INT}(z)$ would
be useful.


\section{Essential knots}\label{sec:essential-knots}

The goal of this section is Theorem~\ref{thm:essential-tau(Lz,L)}
(subsection~\ref{subsubsec:essential-order1-rel-selflinking}) which,
together with Theorem~\ref{thm:tau(Lz,L)} in the introduction,
classifies stable knot concordance in products $F\times S^1$ of an
orientable surface $F\neq S^2$ with the circle. Most of the work
involves characterizing new INT relations which arise when defining
$\tau_1$ for singular concordances of knots which are not
null-homotopic.

Throughout this section we continue to assume that $M$ is an
oriented 3-manifold, with $\pi_1M$ torsion-free, and add the
further assumption that $M$ is irreducible.  We restrict our
attention to knots in such $M$, eventually focusing on the case
where $M=F\times S^1$.

Consider an \emph{essential} knot $K\subset M$ with
$[K]=\gamma\neq 1$ generating a cyclic subgroup $\langle
\gamma\rangle<\pi_1M$, for some choice of whisker on $K$. For any
singular concordance $A\looparrowright M\times I$ of $K$ which
supports an order 1 Whitney tower $\cW$, we want to define
$\tau_1(A)$ as before by associating decorated Y-trees to the
intersections between $A$ and the interiors of the Whitney disks
in $\cW$. Defining $\tau_1$ for such an \emph{essential} annulus
will require two modifications of the previous definition for
inessential annuli: The edge decorations will now be cosets by
$\langle \gamma\rangle$, to account for choices of paths in $A$;
and a reformulation of the INT($A$) relations will account for new
indeterminacies in the choices of Whitney disk boundaries. The
relevant INT($A$) relations will turn out to be determined by
order zero intersections between singular self-concordances which
have latitudes representing generators of the centralizer
$\zeta(\gamma)$ of $\gamma$ in $\pi_1M$.

\subsection{Order zero double coset decorations}\label{subsec:order-zero-coset-decorations}
We start by discussing the order zero intersection invariants that
will be used to define the new INT relations, and index the
equivalence classes of singular concordances which support an
order 1 Whitney tower.

Define the abelian group
$\cT_0(\langle\gamma\rangle\backslash\pi_1M/\langle\gamma\rangle)$
in the same way as $\cT_0(\pi_1M)$ except that the order zero
trees are decorated by (representatives of) double cosets of
$\pi_1M$ by $\langle\gamma\rangle$.

In defining an order zero intersection invariant for annuli, we
will only be concerned with two separate cases: either all
univalent vertices will have the same label; or the pairs of
univalent vertices on all edges (order zero trees) will be labeled
distinctly from a set of two labels. So in the latter case
$\cT_0(\langle\gamma\rangle\backslash\pi_1M/\langle\gamma\rangle)$
can be identified with the $\Z$-span of the double coset space,
and in the first case one takes a further quotient by inversion of
coset representatives and by killing the double coset represented
by the trivial element $1\in\pi_1M$.

If $A$ and $A'$ are singular concordances of knots in $M$
representing $\gamma$, then sheet-changing paths through the
intersections between $A$ and $A'$ determine double cosets, and we
define the \emph{order zero non-repeating intersection tree}
$$
\lambda_0(A,A')
\in\cT_0(\langle\gamma\rangle\backslash\pi_1M/\langle\gamma\rangle)
$$
by summing as usual over the signed intersections between $A$ and
$A'$, and ignoring any self-intersections. $\lambda_0(A,A')$ is
invariant under homotopy (rel $\partial$) and is the complete
obstruction to pairing $A\cap A'$ by Whitney disks (an order 1
\emph{non-repeating Whitney tower}).

Similarly, the \emph{order zero self-intersection tree}
$$
\tau_0(A)
\in\cT_0(\langle\gamma\rangle\backslash\pi_1M/\langle\gamma\rangle)
$$
is defined by summing over all signed self-intersections of $A$,
is invariant under homotopy (rel $\partial$), and is the complete
obstruction to the existence of an order 1 Whitney tower on $A$.

\subsection{Order 1 coset
decorations}\label{subsec:order-1-coset-decorations}

Define the abelian group
$\cT_1(\langle\gamma\rangle\backslash\pi_1M)$ in the same way as
$\cT_1(\pi_1M)$ except that the edges of the Y-tree generators are
decorated by (representatives of) left cosets
$\langle\gamma\rangle\backslash\pi_1M$ -- here we are assuming
that \emph{all edges in our Y-trees are oriented towards the
trivalent vertex}. In this section we only consider knots, so all
univalent vertices in
$\cT_1(\langle\gamma\rangle\backslash\pi_1M)$ have the same label
which is suppressed from notation.

$\cT_1(\langle\gamma\rangle\backslash\pi_1M)$ is the quotient of
$\cT_1(\pi_1M)$ by the relations:
$$
(a,b,c)=(\gamma^q a, \gamma^r b, \gamma^s c) \quad\quad
\mbox{for}\quad q,r,s \in\Z.
$$

For an order 1 Whitney tower $\cW$ on a singular concordance $A$
of $K$ as above, it is clear that these relations are exactly what
is needed to account for choices of sheet-changing paths from $A$
into the Whitney disks $\cW$ when defining $\tau_1(A)$.

\subsection{INT relations for essential
knots}\label{subsec:INT-for-essential-knots} Fixing the boundary
and changing the interior of any Whitney disk in $\cW$ does not
change $\tau_1(\cW)$ since we are assuming that $M$ is
irreducible. Also, by the same argument as in
subsection~\ref{subsec:htpy-invariance-of-tau-proof}, changing the
boundaries of the Whitney disks in $\cW$ by a regular homotopy in
$A$ does not change $\tau_1(\cW)$. However, changing the homotopy
class in $A$ of Whitney disk boundaries can change $\tau_1(\cW)$
as we discuss next.

Consider a Whitney disk $W$ in $\cW$. Changing $\partial W$ by
winding one arc $m$ times around $A$ and the other arc $n$ times
around $A$ creates a closed curve in $A$ representing
$\gamma^m\beta\gamma^n\beta^{-1}\in\pi_1M$, where $\beta$ is
determined by double-point loops through points paired by $W$, and
the signs of the integers $m$ and $n$ correspond to the direction of
winding relative to $\gamma$. This closed curve bounds a Whitney
disk exactly when $\gamma^m\beta\gamma^n\beta^{-1}=1$. Solutions to
such Baumslag-Solitar equations are highly restricted in 3-manifold
groups, for instance if $M$ is sufficiently large then $m=\pm n$ by
Jaco (\cite{Ja}). We make the further assumption that if
$\gamma^m\beta\gamma^n\beta^{-1}=1$, then $\gamma$ and $\beta$
commute, as will be the case for all $\gamma$ when $M=F\times S^1$.
\begin{figure}
\centerline{\includegraphics[width=130mm]{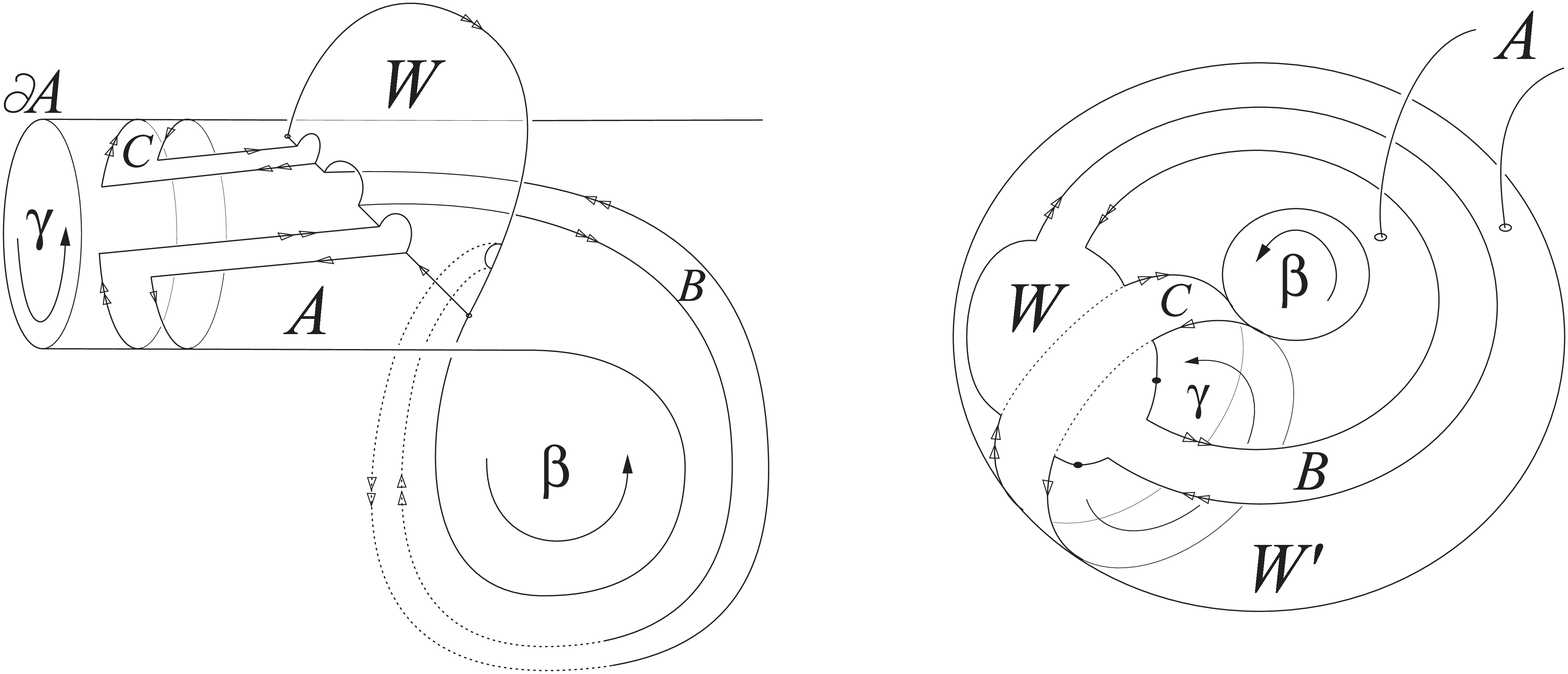}}
         \caption{}
         \label{Essential-W-disk-boundary-change-fig}
\end{figure}

Assuming that $\gamma^m\beta\gamma^n\beta^{-1}=1$ implies
$\gamma\beta\gamma^{-1}\beta^{-1}=1$, it is sufficient to compute
the change in $\tau_1(\cW)$ due to changing a Whitney disk
boundary $\partial W$ by winding one arc around $A$ in the
positive $\gamma$ direction and the other arc around $A$ in the
direction of $\gamma^{-1}$, and replacing $W$ by a Whitney disk
$W'$ with this new boundary. Such a change is illustrated in the
left hand side of
Figure~\ref{Essential-W-disk-boundary-change-fig}, which also
indicates how two thin bands (half-tubes) $B$ and $C$ can be added
to $W$ so that $W\cup B\cup C\cup W'$ forms a torus $T$ (here the
orientation of $W'$ is reversed). Since $A$ can be assumed to miss
$B$ and $C$, the order zero intersections between $A$ and $T$
measure the change in $\tau_1$ as follows.

Take a whisker on $T$ to be a common whisker for $W$ and $W'$
where their boundaries meet. As in the previous INT relations, the
two descending edges on Y-trees for $W$ and $W'$ are decorated by
the same pair of fundamental group elements (now coset
representatives), which here can be normalized to $1$ and $\beta$
by choosing the whisker on $T$ appropriately. The change in
$\tau_1$ is determined by sheet-changing paths from $A$ into $T$
which represent cosets decorating the third edges of the Y-trees.
Note that allowing these paths to cross $B$ is irrelevant since we
are working modulo multiplication by $\gamma$, but we can not
assume to allow paths in $T$ to cross $C$ which would correspond
to multiplication by $\beta$. So the change in $\tau_1$ is truly
measured by order zero intersections between $A$ and the annulus
which is gotten from $T$ essentially by cutting out $C$, as
described next.

The band $C$ can be extended to an annulus $\overline{C}$ by
adding a small embedded band in $W$, as indicated by the dotted
lines in the right-hand side of
Figure~\ref{Essential-W-disk-boundary-change-fig}. Remove
$\overline{C}$ from $T$ and extend the resulting boundary
components by disjointly embedded collars to parallel copies of
$K$ in $M\times\{0\}$ to get an annulus $A'$ which has the same
order zero intersections with $A$ as $T$. Now the change in
$\tau_1$ can be expressed as $(1,\beta,\lambda_0(A,A'))$, where
$\lambda_0$ is the order zero non-repeating intersection tree of
subsection~\ref{subsec:order-zero-coset-decorations} above. Since
$A$ has only cancelling pairs of self-intersections, we can add a
parallel copy of $A$ to $A'$ along one boundary component to get a
singular concordance $A''$ of a parallel copy of $K$ such that
$\lambda_0(A,A'')=\lambda_0(A,A')$. Now we can extend both $A$ and
$A''$ by $-A$ (and a parallel copy) to get pair of singular
\emph{self}-concordances $A_1$ and $A_2$ of $K$ and a parallel
copy of $K$. Again, the intersections between $-A$ and its
parallel copy occur in cancelling pairs, so the change in $\tau_1$
can be expressed as $(1,\beta,\lambda_0(A_1,A_2))$.

By Lemma~\ref{lem:straightening-lemma} and the homotopy invariance
of $\lambda_0$, $\lambda_0(A_1,A_2)$ only depends on the elements
determined by latitudes of $A_1$ and $A_2$. Since a latitude of
$A_1=A-A$ determines the trivial element $1\in\pi_1M$, $A_1$ is
homotopic to the product self-concordance $K\times I$, also by
Lemma~\ref{lem:straightening-lemma}. Since $A'$ has a latitude
representing $\beta$, and $A_2$ was formed by extending $A'$ by
$A-A$, a latitude of $A_2$ determines the element $\beta$ in the
centralizer $\zeta(\gamma)$ of $\gamma$ in $\pi_1M$. Thus, the
possible changes in $\tau_1(\cW)$ are measured by
$\lambda_0(K\times I,A_\beta)$ where $\beta$ ranges over
$\zeta(\gamma)$, and we have the following INT relations.

For $[K]=\gamma\in\pi_1M$, define the INT($K$) relations by:
$$
\mbox{INT}(K)=(1,\beta,\lambda_0(K\times I,A_\beta))
$$
where $\beta$ ranges over $\zeta(\gamma)$ and $A_\beta$ is any
singular self-concordance of $K$ having a latitude representing
$\beta$. These INT relations are completely determined by the
values on a set of generators for $\zeta(\gamma)$ as illustrated
in the next subsection.

It follows from this discussion that (assuming $\gamma$ has only
trivial Baumslag-Solitar solutions) for any singular concordance $A$
of $K$ supporting an order 1 Whitney tower $\cW$
$$
\tau_1(A):=\tau_1(\cW)\in\frac{\cT_1(\langle\gamma\rangle\backslash\pi_1M)}{\mbox{INT}(K)}
$$
does not depend on the choice of $\cW$, and is invariant under
regular homotopy (rel $\partial$) of $A$. Also, the geometric
conditions equivalent to the vanishing of previous versions of
$\tau_1$ hold in this setting since again the relations can all be
realized by geometric constructions.

\subsection{Stable concordance of knots in $F\times
S^1$}\label{subsec:stable-conc-FtimesS1} We now restrict attention
to essential knots in $M=F\times S^1$, the product of a compact
orientable surface $F\neq S^2$ with the circle. The results in
this section, together with Theorem~\ref{thm:tau(Lz,L)} for
null-homotopic knots, give a complete characterization of stable
knot concordance in this setting.

\subsubsection{The order zero relative self-linking
invariant}\label{subsubsec:essential-order-zero-rel-selflinking}
Fix a knot $K_\gamma$ in a non-trivial free homotopy class
together with a whisker such that $[K_\gamma]=\gamma\in\pi_1M$.
For any singular concordance $A$ from $K_\gamma$ to any knot $K$,
define the \emph{order zero relative self-intersection tree}
$\tau_0(K_\gamma,K)$ by:
$$
\tau_0(K_\gamma,K):=\tau_0(A)\in\cT_0(\langle\gamma\rangle\backslash\pi_1M/\langle\gamma\rangle).
$$
By \cite{S}, it is possible to choose $K_\gamma$ in each free
homotopy class so that the order zero relative self-linking with
$K_\gamma$ classifies order 1 Whitney concordance, meaning that
$\tau_0(K_\gamma,K)=\tau_0(K_\gamma,K')$ if and only if $K$ and
$K'$ are order 1 Whitney concordant, for any knots $K$ and $K'$
(freely) homotopic to $K_\gamma$. (This invariant was denoted
$\mu(K_\gamma,K)$ in \cite{S}.) Here equality is modulo
conjugation by the centralizer $\zeta(\gamma)$, i.e. modulo the
restriction to $\zeta(\gamma)$ of the $\Psi$-action on
$\cT_0(\langle\gamma\rangle\backslash\pi_1M/\langle\gamma\rangle)$.
Note that choices of whiskers on $K_\gamma$ giving
$[K_\gamma]=\gamma$ differ by loops representing elements in
$\zeta(\gamma)$, and all latitudes of singular self-concordances
of $K_\gamma$ determine elements in $\zeta(\gamma)$.

As explained in \cite{S}, the key property required of $K_\gamma$
in order for $\tau_0(K_\gamma,\cdot)$ to classify order 1 Whitney
concordance is that
$\tau_0(H)=0\in\cT_0(\langle\gamma\rangle\backslash\pi_1M/\langle\gamma\rangle)$
for any singular self-concordance $H$ of $K_\gamma$. (In the
present case $M$ is irreducible; in general the required vanishing
of $\tau_0(H)$ would be modulo intersections with 2-spheres --
hence the terminology ``$K_\gamma$ is spherical'' in \cite{S}.)
This key property can be checked on self-homotopies having
latitudes representing a set of generators for $\zeta(\gamma)$,
and can be easily satisfied by choosing $K_\gamma$ as follows.

If $\gamma$ is carried by (a multiple of) a circle fiber, then take
$K_\gamma$ to lie in a tubular neighborhood of a fiber. In this
case, $\zeta(\gamma)=\pi_1M$, and $K_\gamma$ can be isotoped around
any generating loop in $M$. If $\gamma$ is not carried by a circle
fiber, then take $K_\gamma$ in a tubular neighborhood of a knot
$K_\rho$ representing a primitive element $\rho\in\pi_1M$, with
$\gamma=\rho^n$ for some natural number $n$. In this case
$\zeta(\gamma)$ will be isomorphic to $\Z\oplus\Z$ generated by
$\rho$ and a fiber element, unless $F$ is either a 2-torus or
2-disk, and in all cases $K_\gamma$ can be isotoped around loops
representing generators of $\zeta(\gamma)$.

\subsubsection{The order 1 relative self-linking
invariant}\label{subsubsec:essential-order1-rel-selflinking} Fix
now $K_\gamma$ as above, which classifies order 1 Whitney
concordance in a non-trivial free homotopy class represented by
$[K_\gamma]=\gamma\in\pi_1M$. For any
$z\in\cT_0(\langle\gamma\rangle\backslash\pi_1M/\langle\gamma\rangle)$,
construct a clasp knot $K_z$ such that $\tau_0(K_\gamma,K_z)=z$ by
introducing clasps into $K_\gamma$ (just as was done for
null-homotopic knots). For any $K$ which is order 1 Whitney
concordant to $K_z$ (equivalently $\tau_0(K_\gamma,K)=z$), define
the \emph{order 1 relative self-intersection tree} $\tau_1(K_z,K)$
by:
$$
\tau_1(K_z,K):=\tau_1(A)\in\frac{\cT_1(\langle\gamma\rangle\backslash\pi_1M)}{\mbox{INT}(z)}
$$
where $A$ is any concordance from $K_z$ to $K$, and the INT$(z)$
relations are:
$$
0=(1,\beta^r,(1-\beta^r)(z+\overline{z}))
$$
with $\beta$ ranging over a generating set for $\zeta(\gamma)$,
and $r\in\Z$.

\begin{figure}
\centerline{\includegraphics[width=90mm]{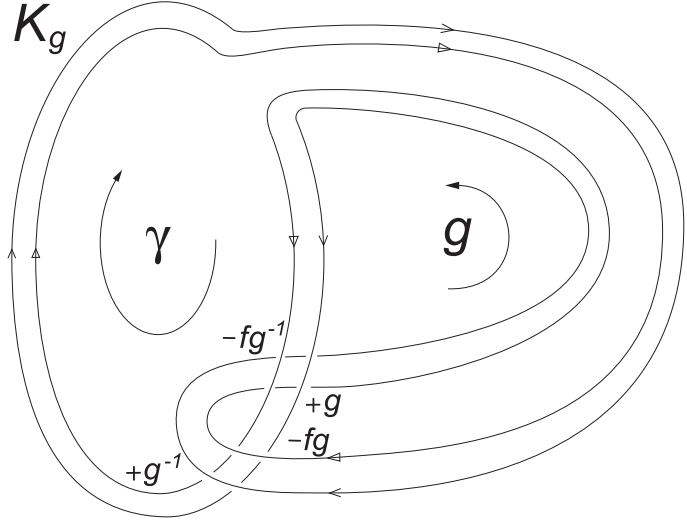}}
         \caption{A knot $K_g$ together with a parallel copy sitting in
         a slice $F\times I$ of $F\times S^1$. A self-isotopy $A_f$ of the parallel copy
         up around the circle direction creates the indicated crossing-change intersections,
         giving $\lambda_0(K_g\times I,A_f)=g+g^{-1}-fg-fg^{-1}$.}
         \label{Essential-knot-htpy-fig}
\end{figure}
To see that these are the correct INT relations
($\mbox{INT}(z)=\mbox{INT}(K_z)$ as defined above in
subsection~\ref{subsec:INT-for-essential-knots}) start by
considering Figure~\ref{Essential-knot-htpy-fig}, which illustrates
the case where $\gamma$ is primitive and $z$ is represented by a
single element $g\in\pi_1(F\times S^1)\cong \pi_1F\times\langle
f\rangle$, with both $\gamma$ and $g$ contained in the subgroup
$\pi_1F$ carried by a thickened surface slice which in the figure is
represented by the plane of the paper cross an interval. Isotoping a
parallel push-off of $K_g$ up around the circle direction describes
an annulus $A_f$ which has a latitude representing the circle
generator $f$, and the order zero non-repeating intersection tree
$$
\lambda_0(K_g\times I,A_f)=g+g^{-1}-fg-fg^{-1}=(1-f)(g+g^{-1})
$$
can be computed from crossing changes in the figure. Note that
introducing twists between $K_g$ and its parallel push-off only
creates terms of the form $\pm(1-f)$ via this construction which
do not contribute to INT$(g)$ because the 2-torsion elements
$(1,f,(1-f))=(1,f,1)-(1,f,f)$ are zero in $\cT_1$ by the AS and FR
relations.

Iterating this construction confirms the relations
$(1,f^r,(1-f^r)(g+g^{-1}))$ for the generator $f\in\zeta(\gamma)$
in this case. It is easy to see that the same computation holds if
$g$ replaced by any sum $z=\sum g_i$, even if $g_i$ are allowed to
contain powers of $f$ as factors (with corresponding clasps guided
by arcs wrapping around circle fibers). One can also check that
allowing $\gamma$ to contain factors of $f$ does not change the
computation.

Under the assumption that $\gamma$ is primitive, the computation
so far verifies the formula for the INT($z$) relations in the
``generic'' case where $\zeta(\gamma)\cong\Z\oplus\Z$ is generated
by $f$ and $\gamma$, since in this case self-homotopies with
latitudes representing $\gamma$ give only trivial relations: A
parallel push-off of the product annulus $K_z\times I$ is disjoint
from $K_z\times I$ and has a latitude representing $\gamma$.

Consider now the case where $\gamma=f$, so that $K_z$ is created
by introducing clasps into a circle fiber $K_f$. In this case $f$
is central, and for any $\beta\in\pi_1M$, a self-homotopy
$A_\beta$ of $K_z$ with a latitude representing $\beta$ and
$\lambda_0(K_z\times I,A_\beta)=(1-\beta)(z+\overline{z})$ can be
constructed as follows. Undoing the clasps on a parallel copy
$K'_z$ of $K_z$ describes a homotopy of $K'_z$ to a knot $K'_f$
which may twist around the circle fiber $K_f$. The homotopy so far
creates intersections with $K_z\times I$ of the form
$z+\overline{z}$ (half of the crossing changes in
Figure~\ref{Essential-knot-htpy-fig} for each clasp). Next, $K_f'$
can be isotoped around $M$ so that a basepoint on $K_f'$ traces
out a loop representing $\beta$; this self isotopy of $K_f'$ only
creates intersections with $K_z\times I$ of the form
$\pm(1-\beta)$ corresponding to any twists between $K_f'$ and
$K_f$ which are pulled apart and then recreated; as mentioned
above these intersections do not contribute to INT($z$). Now
reintroducing the clasps into $K_f'$ describes the last part of
the self-homotopy $A_\beta$ of $K_z$, and creates intersections of
the form $-\beta(z+\overline{z})$, so that $\lambda_0(K_z\times
I,A_\beta)=(1-\beta)(z+\overline{z})$ as desired.

The cases where $\gamma$ is not primitive are checked by
constructing self-homotopies similar to the case $\gamma=f$ in the
previous paragraph. For instance, if $\gamma=f^n$ then $K_z$ is
created from $K_{f^n}$ which is contained in a tubular
neighborhood of a circle fiber and the rest of the construction as
for the case $n=1$ applies. The same construction applies for
$\gamma=\rho^n$, with $\rho$ a primitive generator of
$\zeta(\gamma)$ and $K_z$ created from $K_{\rho^n}$ which is
contained in a tubular neighborhood of $K_\rho$; and also to any
$\gamma$ in the non-generic cases where $F$ is a torus or a disk.

Combined with Theorem~\ref{thm:tau(Lz,L)}, the following theorem
gives a classification of stable knot concordance in $F\times
S^1$:
\begin{thm}\label{thm:essential-tau(Lz,L)}
For $K_z\subset M=F\times S^1$ as above, the following are
equivalent:
\begin{enumerate}
  \item $\tau_1(K_z,K)=\tau_1(K_z,K')$.
  \item $K$ and $K'$ are stably concordant.
  \item $K$ and $K'$ are order 2 Whitney concordant.
  \item $K$ and $K'$ are height 1 Whitney concordant.
\end{enumerate}
\end{thm}
\begin{proof} The proof of Theorem~\ref{thm:essential-tau(Lz,L)} follows the
arguments in subsections~\ref{subsec:tauLzL-well-defined} and
\ref{subsec:LzLstatements-proof} in the proof of
Theorem~\ref{thm:tau(Lz,L)}. The key point needed is that $\tau_1$
vanishes on all self-homotopies of $K_z$ which support an order 1
Whitney tower (equivalently, have vanishing $\tau_0$). If $H_\psi$
is a self-homotopy of $K_z$ with a latitude representing $\psi$,
then $\psi\in\zeta(\gamma)$, and by
Lemma~\ref{lem:straightening-lemma} $\tau_0(H_\psi)$ only depends
on $\psi$. It follows that $\tau_0(H_\psi)=z-\psi z\psi^{-1}$,
since such $H_\psi$ can be constructed by undoing the clasps on
$K_z$ to get $K_\gamma$, then isotoping $K_\gamma$ around
$\psi\in\zeta(\gamma)$, then redoing the clasps to return to
$K_z$. Thus, such an $H_\psi$ admits an order 1 Whitney tower if
and only if $\tau_0(H_\psi)$ vanishes if and only if
$\psi\in\Psi_z\cap\zeta(\gamma)$. In the current setting, $\Psi_z$
is always untwisted, and it can be easily checked that for
$\psi\in\Psi_z\cap\zeta(\gamma)$ the self-isotopies of $K_\gamma$
realizing $\psi$ (as described above in
subsection~\ref{subsubsec:essential-order-zero-rel-selflinking})
extend to self-isotopies of $K_z$.
\end{proof}



\end{document}